\documentclass{amsart}

\usepackage{amsmath}   
\usepackage{amsthm}
\usepackage{amssymb}
\usepackage{longtable}
\usepackage[table,dvipsnames]{xcolor} 
\usepackage{imakeidx}
\makeindex 

 \usepackage[msc-links]{amsrefs}

\usepackage[capitalise]{cleveref}
\crefname{thm}{Thm.}{}
\crefname{prop}{Prop.}{}
\crefname{lem}{Lem.}{}
\crefname{cor}{Cor.}{}
\crefname{equation}{Eq.}{}
\crefname{exa}{Exa.}{}
\crefname{defi}{Def.}{}
\usepackage[shortlabels]{enumitem}
\newtheorem{thm}{Theorem}
\newtheorem{prop}{Proposition}
\newtheorem{lem}{Lemma}
\newtheorem{defi}{Definition}
\newtheorem{exa}{Example}
\newtheorem{rem}{Remark}
\newtheorem{cor}{Corollary}

\DeclareMathOperator{\coeff}{coeff}
\DeclareMathOperator{\Ch}{{\mathfrak{C}}}       
\DeclareMathOperator{\SU}{SU}

\DeclareMathOperator{\lcm}{lcm}
\DeclareMathOperator{\Spec}{Spec}
\DeclareMathOperator{\Proj}{Proj}
\DeclareMathOperator{\Pic}{Pic}
\DeclareMathOperator{\SL}{SL}
\DeclareMathOperator{\GL}{GL}
\DeclareMathOperator{\PGL}{PGL}

\DeclareMathOperator{\Hmult}{H}
\DeclareMathOperator{\hlog}{h}

\newcommand\chowh{\hlog_{\Ch}}   

\DeclareMathOperator{\wh}{\Hmult_{\w}} 
\DeclareMathOperator{\lwh}{\hlog_{\w}} 

\DeclareMathOperator{\ih}{\mathfrak{H}} 
\DeclareMathOperator{\lih}{\mathfrak{h}} 
\newcommand\cih{\mathfrak{h}^{\Ch}}   

\def\cR{\mathcal{R}}
\def\div{\mbox{div}}

\def\mod{\mbox{mod }}
\def\deg{\mbox{deg }}

\def\supp{\mbox{supp}}
\def\Sym{\mbox{Sym}}

\newcommand\x{\mathbf{x}}
\newcommand{\F}{\mathbb{F}}

\newcommand{\Z}{\mathbb{Z}}
\newcommand{\Q}{\mathbb{Q}}
\newcommand{\R}{\mathbb{R}}
\newcommand{\C}{\mathbb{C}}
\newcommand{\bP}{\mathbb{P}}
\newcommand{\A}{\mathbb{A}}

\newcommand{\cH}{\mathcal{H}}
\newcommand{\M}{\mathcal{M}}

\newcommand{\X}{\mathcal{X}}
\newcommand{\B}{\mathcal{B}}
\newcommand{\cC}{\mathcal{C}}
\newcommand{\cL}{\mathcal{L}}

\newcommand{\cZ}{\mathcal{Z}}

 \newcommand{\cE}{\mathcal{E}}
\newcommand{\cN}{\mathcal{N}}
\newcommand{\cM}{\mathcal{M}}

\newcommand{\cU}{\mathcal{U}}
\newcommand{\cO}{\mathcal{O}}

\newcommand{\hn}{\mathfrak{h}}

\newcommand{\w}{\mathfrak{w}}

\newcommand{\f}{\mathfrak{f}}

\newcommand\bm{\mathbf{m}}

\newcommand{\il}{\zeta}

\newcommand{\Orb}{\mathrm{Orb}}

\title{Weighted Heights and GIT Heights}

\author{E. Shaska}
\address{Department of Computer Science \\ 
College of Computer Science and Engineering \\
Oakland  University,  Rochester, MI, 48309. }
\email{elirashaska@oakland.edu}

\author{T. Shaska}
\address{Department of Mathematics and Statistics \\
College of Arts and Sciences \\
 Oakland University, Rochester, MI, 48309}
\email{tanush@umich.edu} 
 
\begin{document}

\begin{abstract}
We investigate the relationship between Geometric Invariant Theory (GIT) heights and weighted heights, with a focus on their interaction in weighted projective spaces and their application to binary forms. Building on the weighted height framework developed in \cite{b-g-sh, s-sh}, we relate it to Zhang’s GIT height via the Veronese map. For a semistable cycle \( \cZ \subset \bP_{\w, \overline{\Q}}^N \), we show that the GIT height decomposes into the logarithmic weighted height plus an Archimedean correction from the Chow metric. Specializing to degree-\( d \) binary forms \( f \in V_d \), we define an invariant height \( \cih(f) \) with respect to the Chow metric and prove that the moduli weighted height \( \lwh(\xi(f)) \) of \( f \)'s invariants satisfies
\[
\lwh(\xi(f)) = \frac{1}{[K:\Q]} \, \cih(f) + \chowh(f)
\]
thereby connecting GIT stability, moduli theory, and arithmetic complexity in a unified framework.
\end{abstract}



 \maketitle              

\section{Introduction}\label{sec:intro}
Height functions lie at the heart of arithmetic geometry. They quantify the arithmetic complexity of algebraic varieties and play a central role in Diophantine geometry, Arakelov theory, and the study of moduli spaces. Classical height constructions—such as the Weil height on projective space and the Néron–Tate height on abelian varieties—provide powerful tools for measuring the size of algebraic points and cycles. However, in emerging settings such as weighted projective spaces, where coordinates carry distinct degrees, classical heights must be refined to capture the graded structure faithfully and to reflect both arithmetic and geometric data.

This paper investigates the interplay between two height frameworks that address this challenge: the GIT height introduced by Zhang~\cite{zhang}, which encodes stability in moduli problems via Geometric Invariant Theory, and the weighted height developed in our earlier work~\cite{b-g-sh, s-sh}, which generalizes the Weil height to weighted projective spaces. Our goal is to establish a conceptual and computational bridge between these two perspectives, with the moduli space of binary forms \( \B_d \) serving as the guiding example.

The moduli space \( \B_d \), parameterizing degree-\( d \) binary forms up to \( \SL_2 \)-action, embeds naturally into a weighted projective space via its ring of invariants. This setting lies at the intersection of invariant theory and weighted geometry. On one hand, the GIT height captures the intrinsic arithmetic size of semistable orbits and reflects moduli stability; on the other, the weighted height measures complexity through the geometry of the weighted embedding. A central theme of this work is to show that these two notions are compatible and that their interaction reveals refined arithmetic information.

We show that for semistable cycles \( \X \subset \bP_{\w, \overline{\Q}}^N \), the logarithmic weighted height \( \lwh(\X) \) relates to the GIT height under the Veronese map \( \phi_m \) by the formula  
\(
 \lwh(\X) = \frac{(r+1)d}{m} \, \hat{\hlog}(\phi_m(\X)),
\)
as established in \cref{lem:git-weighted} and refined in \cref{thm:git-weighted-relation} by including the Archimedean contribution from the Chow metric.  
Specializing to binary forms \( f \in V_d \), we define an invariant height \( \cih(f) \), adapting Zhang’s construction to weighted coordinates.  
Our main result, \cref{thm:weighted-invariant}, shows that the moduli weighted height \( \lwh(\xi(f)) \) of the invariants of \( f \) satisfies  
\(
 \lwh(\xi(f)) = \frac{1}{[K:\Q]} \, \cih(f) + \chowh(f),
\)
where \( \chowh(f) \) is the Chow height.  
This decomposition isolates the canonical, GIT-invariant component from the metric-dependent term at infinity, revealing a clean arithmetic structure underlying the moduli class of \( f \).

Concrete computations, such as the explicit formula  
$
 \cih(f) = \tfrac{d}{2}\log(1 + |a_0|^2) 
$
for binary forms \( f = x^d - a_0 y^d \) (\cref{lem:power_form}), illustrate the practical consequences of this theory.  
These examples extend computational insights from earlier works (e.g.,~\cite{rabina, r-savin}) and show how classical moduli problems for binary forms can be reinterpreted within a modern arithmetic framework.  
The resulting picture suggests that heights on moduli spaces naturally decompose into invariant-theoretic and metric contributions—a phenomenon familiar from Arakelov geometry, now visible in the weighted setting.  
The same framework extends to binary forms with non-trivial automorphism groups, where root orbits under \( G \subset \PGL_2 \) yield refined expressions for the Chow–invariant height (\cref{thm:automorphism_height}).

\cref{sec:git-invariant} develops the GIT–weighted correspondence via the Veronese map and the Deligne pairing,
while \cref{sec:binary_forms} applies it to binary forms, computing explicit Chow–invariant heights and proving the main decomposition of the moduli height.
\cref{sec:binary_forms} also introduces Chow coordinates for root divisors, studies naive and minimal heights, and analyzes both non-Archimedean and Archimedean contributions in terms of semistability and \( \SU(2) \)-invariant metrics.  
The global invariant height \( \cih(f) \) unifies these contributions, with explicit computations for cubics, power forms, and forms with automorphisms.  
The main \cref{thm:weighted-invariant} provides the final decomposition of the moduli height, separating its intrinsic and Archimedean parts.

Finally, \cref{sec:conclusion} summarizes the results and outlines directions for future research.  
In summary, this paper contributes to arithmetic invariant theory by establishing a precise correspondence between GIT and weighted heights for binary forms and their moduli.  
The decomposition of moduli height into canonical and Archimedean components opens a new pathway toward understanding the arithmetic geometry of weighted projective varieties and moduli spaces and lays a foundation for broader generalizations in height theory.

\smallskip
\noindent \textbf{Notation:} 
Throughout this paper \(k\) denotes a number field, \(K\) a finite extension of \(k\), 
\(\cL\) a line bundle on a variety \(\X\),
\(s\) a section of \(\cL\), and \(\widehat{\cL}\) a metrized line bundle endowed with a specific norm. 
If \(\Hmult\) denotes a multiplicative height, its logarithmic analogue is denoted by \(h=\log \Hmult\), 
while their weighted counterparts are denoted by \(\wh\) and \(\lwh\), respectively.  
The set of places of \(K\) is \(M_K = M_K^0 \cup M_K^\infty\), with local degrees \(n_\nu=[K_\nu:\Q_\nu]\).  
Weighted projective spaces are written \(\bP_{\w}^N = \bP(q_0,\dots,q_N)\), with least common multiple 
\(m=\mathrm{lcm}(q_0,\dots,q_N)\); the corresponding Veronese map \(\phi_m\) will be used to compare 
weighted and standard projective embeddings.  
For binary forms \(f\in V_d=\Sym^d(K^2)\), we denote by 
\(\xi(f)=[\xi_0(f):\dots:\xi_n(f)]\in \bP_{\w}^n(K)\) its image under the GIT quotient map, 
where \(\w=(\deg \xi_0,\dots,\deg \xi_n)\).

\section{Preliminaries}\label{sec:preliminaries}


Let \( \X\) \label{variety}
 be a variety over a field \( \F \) and \( E \) a vector bundle of rank \( r \) over \( \X \) defined  over \( \F \) with the equipped  morphism \( \pi_E : E \to \X \).
A \emph{section} of a vector bundle \( E \) over an open set \( U \subset \X \) is a morphism \( s : U \to E \) such that \( \pi_E \circ s = \mathrm{id}_U \). If \( U = \X \), then \( s \) is a \emph{global section}. The set of sections over \( U \), denoted \( \Gamma(U, E) \), forms a \( \F \)-vector space. The \emph{sheaf of sections} \( \cE \) is defined by \( \cE(U) = \Gamma(U, E) \) for each open \( U \subset \X \), providing a sheaf-theoretic framework for studying sections locally and globally.

A \emph{line bundle}\label{line bundle cL}  \index{$\cL$}
\( \cL \) on \( \X \) is a vector bundle of rank 1. For \( n \in \Z \), the \( n \)-fold tensor power is denoted \( \cL^{\otimes n} \), with conventions \( \cL^{\otimes 0} = \cO_\X \) (the structure sheaf) and \( \cL^{\otimes (-n)} = (\cL^*)^{\otimes n} \), where \( \cL^* \) is the dual bundle. The \emph{Picard group} \( \Pic(\X) \) is the group of isomorphism classes of line bundles under the tensor product \( \otimes \), with the inverse of \( [\cL] \) given by \( [\cL^*] \).

Consider the trivial bundle \( E = \bP_\F^n \times \A_\F^{n+1} \), with coordinates \( [x_0 : \cdots : x_n] \) on \( \bP_\F^n \) and \( (y_0, \ldots, y_n) \) on \( \A_\F^{n+1} \). The \emph{tautological line bundle} \( \cO_{\bP_\F^n}(-1) \) is the subvariety \( \cL \subset E \) defined by the equations
\(
x_i y_j - x_j y_i = 0
\)
for all $0 \leq i, j \leq n$.

The projection \( \pi_\cL : \cL \to \bP_\F^n \) is a morphism, and local trivializations over \( U_\alpha = \{x_\alpha \neq 0\} \) are given by
$\phi_\alpha : \pi_\cL^{-1}(U_\alpha) \to U_\alpha \times \A_\F^1$,  $(\x, \mathbf{y}) \mapsto \left( \x, \frac{y_\alpha}{x_\alpha} \right)$, 
with transition functions \( M_{\alpha \beta}(\x) = \frac{x_\beta}{x_\alpha} \). For \( m \in \Z \), define \( \cO_{\bP_\F^n}(m) = \cO_{\bP_\F^n}(-1)^{\otimes (-m)} \), whose sections correspond to homogeneous polynomials of degree \( m \).

Let \( k \) \label{number field $k$}
be a number field of degree \( [k : \Q] = m \), with ring of integers \( \cO_k \), and let \( \overline{k} \) be an algebraic closure of \( k \). A variety \( \X \) over \( k \) is an integral separated scheme of finite type over \( \Spec(k) \), with \( \X(\overline{k}) \) denoting its \( \overline{k} \)-points and \( \X(k) \) its \( k \)-rational points.
The set of places \( M_k \) of \( k \) comprises:   \label{M_k}
i)  Non-Archimedean places \( M_k^0 \), corresponding to primes \( \mathfrak{p} \subset \cO_k \),  
ii)  Archimedean places \( M_k^\infty \), corresponding to    \( \sigma : k \hookrightarrow \C \) up to conjugation.  
For a place \( \nu \in M_k \), the \emph{local degree} \label{local degree}
is \( n_\nu = [k_\nu : \Q_\nu] \),  
where \( k_\nu \) is the completion of \( k \) at \( \nu \). The absolute value \( |\cdot|_\nu \) is normalized to extend the standard absolute value on \( \Q \). 

Let \( M = M_{\overline{k}} \) be the set of places of \( \overline{k} \) extending \( M_k \). An \emph{\( M_k \)-constant} is a function \( \gamma : M_k \to \R \) such that \( \gamma(\nu) = 0 \) for all but finitely many \( \nu \), extended to \( M \) by \( \gamma(v) = \gamma(v|_k) \). An \emph{\( M_k \)-function} on \( \X \) is a map \( \lambda : \X \times M \to \R \) where \( \lambda(\x, v) \) is either an \( M_k \)-constant or infinite.
Two \( M_k \)-functions \( \lambda_1, \lambda_2 \) are \emph{equivalent}, written \( \lambda_1 \sim \lambda_2 \), if there exists an \( M_k \)-constant \( \gamma \) such that
\(
|\lambda_1(\x, v) - \lambda_2(\x, v)| \leq \gamma(v)
\)
for all   \((\x, v) \in \X \times M\).
An \( M_k \)-function \( \lambda \) is \emph{\( M_k \)-bounded} if \( \lambda \sim 0 \).

For an affine variety \( \X \), \label{affine variety}
a set \( E \subset \X \times M \) is an \emph{affine \( M_k \)-bounded set} if there exist coordinates \( x_1, \ldots, x_n \) on \( \X \) and an \( M_k \)-bounded constant \( \gamma \) such that
\(
|x_i(\x)|_v \leq e^{\gamma(v)} 
\)
for $1 \leq i \leq n, \, (\x, v) \in E$. 
For a general variety, \( E \subset \X \times M \) is \emph{\( M_k \)-bounded} if it is covered by finitely many affine \( M_k \)-bounded sets.

A function \( \lambda : \X \times M \to \R \) is \emph{locally \( M_k \)-bounded above} if, for every \( M_k \)-bounded set \( E \subset \X \times M \), there exists an \( M_k \)-constant \( \gamma \) such that \( \lambda(\x, v) \leq \gamma(v) \) on \( E \). It is \emph{locally \( M_k \)-bounded} if both \( \lambda \) and \( -\lambda \) are locally \( M_k \)-bounded above.

Let \( \mathcal{T} \) be a topological space, \( \cL \) a line bundle on \( \mathcal{T} \), \( U \subset \mathcal{T} \) an open set, 
 \( s \) \label{section}
a \emph{section of \( \cL \) over \( U \)}, 
 and   \( \|\cdot\| \) 
a \emph{metric}   \label{metric}
on \( \cL \).
A \emph{metrized line bundle} \label{metrized line bundle}
is a pair \( \widehat{\cL} = (\cL, \|\cdot\|) \).

For a number field \( k \), an \( \cO_k \)-scheme \( \X \), and an embedding \( \sigma : \cO_k \hookrightarrow \C \), the induced map \( \sigma^* : \Spec(\C) \to \Spec(\cO_k) \) equips \( \X_\sigma(\C) \) with a complex manifold structure. Metrics on \( \cL \) are required to be conjugation-invariant at Archimedean places.
A metric is \emph{smooth} if \( \|s\|_U \) is \( \cC^\infty \) for smooth, non-vanishing sections \( s \). A metrized bundle \( \bar{\cL} \) is \emph{Hermitian} if its metric is smooth.

\begin{exa}
For \( \X = \bP_\C^n \) and \( \cL = \cO(1) \), the \emph{standard metric}  and the \emph{Fubini-Study metric}
\begin{equation}\label{exa:standard-metric}
\|s\|(\x) = \frac{|s(\x)|}{\max_i |x_i|},
\quad 
\|s\|_f (\x) = \frac{|s(\x)|}{\left( \sum_{i=0}^n |x_i|^2 \right)^{1/2}},
\end{equation}
are locally bounded (see \cite[Example 2.7.4]{bombieri}).
\end{exa}

Every line bundle on a variety over \( k \) admits a locally bounded metric (see \cite[Prop. 2.7.5]{bombieri}).
For a line bundle \( \cL \) on a flat proper reduced scheme \( \X \) over \( \cO_k \), with generic fiber \( L = \cL_k \) on \( X = \X_k \), the \emph{natural \( M \)-metric} \( \|\cdot\|_\cL \) is defined by setting \( \|s(x)\|_{\cL, u} = 1 \) for a constant section \( s \) at each place \( u \in M \), using the integral model to trivialize locally.

\begin{lem}
The natural \( M \)-metric \( \|\cdot\|_\cL \) is well-defined and locally bounded.
\end{lem}

For a proof see  \cite[Prop. 2.7.5]{bombieri}.

\begin{exa}
For the constant section \( s(\x) = 1 \), the standard metric  \label{standard metric}
and  the Fubini-Study metric   \label{Fubini-Study metric} 
are 
\begin{equation}
\|s\|(\x) = \frac{1}{\max_i |x_i|},
\quad 
\|s\|_f (\x) = \frac{1}{\left( \sum_{i=0}^n |x_i|^2 \right)^{1/2}}.
\end{equation}
\end{exa}

\subsection{Heights on Projective Spaces}\label{sec:heights_proj}

Let \( \x = [x_0 : \cdots : x_n] \in \bP^n(k) \), where \( k \) is a number field. For a place \( \nu \in M_k \), define the local multiplicative and logarithmic heights relative to a metrized line bundle \( \cL = (\cO_{\bP^n_k}(1), \|\cdot\|) \) as
\[
\Hmult_{k, \nu}(\x) := \|s(\x)\|_\nu^{-1} \quad \text{and} \quad h_{k, \nu}(\x) := -\log \|s(\x)\|_\nu,
\]
where \( s \) is a section of \( \cL \) such that \( s(\x) \neq 0 \). Using the standard metric from \cref{exa:standard-metric} with a constant global section \( s(\x) = 1 \), we obtain
\[
\Hmult_\nu(\x) = \left( \max_{0 \leq i \leq n} \{ |x_i|_\nu \} \right)^{-1} \quad \text{and} \quad h_\nu(\x) = -\log \max_{0 \leq i \leq n} \{ |x_i|_\nu \}.
\]
The \emph{multiplicative height} and \emph{logarithmic height} of \( \x \) are defined as
\begin{equation}\label{eq:proj-height}
\Hmult_k(\x) = \prod_{\nu \in M_k} \max_{0 \leq i \leq n} \{ |x_i|_\nu \}^{n_\nu} \quad \text{and} \quad h_k(\x) = \sum_{\nu \in M_k} n_\nu \log \max_{0 \leq i \leq n} \{ |x_i|_\nu \},
\end{equation}
where \( n_\nu = [k_\nu : \Q_\nu] \) is the local degree at \( \nu \). For a finite extension \( K/k \), normalize \( |\cdot|_w \) for \( w \in M_K \) such that \( |x|_w = |x|_\nu^{n_w / n_\nu} \) for \( x \in k \), where \( w|_k = \nu \). Thus,
\[
\Hmult_k(\x) = \Hmult_K(\x)^{1/[K:k]} \quad \text{and} \quad h_k(\x) = \frac{1}{[K:k]} h_K(\x).
\]
%

Given a Cartier divisor \( D = \{(U_i, f_i)\} \) on a variety \( \X \subset \bP_{\overline{k}}^n \), the line bundle \( \cL_D = \cO_\X(D) \) is constructed by gluing
\[
\cL_D|_{U_i} = f_i^{-1} \cO_\X(U_i),
\]
with a canonical section \( g_D \) (the constant section 1 on each patch, adjusted by \( f_i \)). Equip \( \cL_D \) with a locally bounded \( M \)-metric \( \|\cdot\| \) to form the metrized bundle
$\widehat{D} = (\cL_D, \|\cdot\|).$
The \emph{local Weil height} with respect to \( \widehat{D} \) at \( \nu \in M_k \) is
\begin{equation}\label{eq:lwh1}
\lambda_{\widehat{D}}(\x, \nu) = -\log \|g_D(\x)\|_v, \quad \x \in \X \setminus \supp(D),
\end{equation}
where \( v \in M \) satisfies \( v|_k = \nu \).

\begin{exa}\label{exa:local_weil}
For \( \X = \bP_{\overline{k}}^n \) and \( D \) a hyperplane defined by \( \ell(\x) = a_0 x_0 + \cdots + a_n x_n \), the standard metric on \( \cO_\X(1) \) is
\begin{equation}\label{eq:local_weil_metric}
\|\ell(\x)\|_v = \frac{|\ell(\x)|_v}{\max_{0 \leq i \leq n} |x_i|_v},
\end{equation}
which is locally \( M_k \)-bounded on \( U_i = \{x_i \neq 0\} \) since \( \|x_i(\x)\|_v \leq 1 \). Thus, with \( g_D = \ell \),
\begin{equation}\label{eq:lwh2}
\lambda_{\widehat{D}}(\x, \nu) = -\log \frac{|\ell(\x)|_v}{\max_{0 \leq i \leq n} |x_i|_v}.
\end{equation}
\end{exa}


For a variety \( \X \subset \bP_{\overline{k}}^n \) over \( k \) and a line bundle \( \cL \) on \( \X \), consider the metrized line bundle
\[
\widehat{\cL} = (\cL, (\|\cdot\|_v)_{v \in M}) \in \widehat{\Pic}(\X).
\]
For \( \x \in \X \), let \( K = k(\x) \) be the field of definition. For each \( u \in M_K \), choose \( v \in M \) with \( v|_k = u \) and define
\[
\|\cdot\|_u = \|\cdot\|_v^{n_u / [K:k]},
\]
where \( n_u = [K_u : \Q_u] \). This is independent of the choice of \( v \) by the metric’s properties. Take an invertible section \( g \) of \( \cL \) with \( \x \notin \supp(\div(g)) \), and form
\[
\widehat{\cL}_g = (\cO_\X(\div(g)), (\|\cdot\|_u)).
\]
The \emph{global Weil height} is
\begin{equation}\label{eq:gwh}
h_{\widehat{\cL}}(\x) = \frac{1}{[K:k]} \sum_{u \in M_K} \lambda_{\widehat{\cL}_g}(\x, u),
\end{equation}
where \( \lambda_{\widehat{\cL}_g}(\x, u) = -\log \|g(\x)\|_u \). This is independent of \( K \) and \( g \).

\subsection{Weighted Varieties and Their Heights}\label{sec-3}
Let \( k \) be a number field with ring of integers \( \cO_k \), and  \( \overline{k} \) be its algebraic closure. 
A \emph{weighted tuple} in \( \cO_k^{n+1} \) is a tuple \( \x = (x_0, \ldots, x_n) \) with weights \( \w = (q_0, \ldots, q_n) \), where each \( q_i \) is a positive integer. Scalar multiplication by \( \lambda \in k^\ast \) is defined as
\[
\lambda \star (x_0, \ldots, x_n) = (\lambda^{q_0} x_0, \ldots, \lambda^{q_n} x_n).
\]
The quotient under this action is the \emph{weighted projective space} \( \bP_{\w, k}^n \), with points \( [x_0 : \cdots : x_n] \) where \( (x_0, \ldots, x_n) \sim (\lambda^{q_0} x_0, \ldots, \lambda^{q_n} x_n) \) for \( \lambda \in k^\ast \). Heights on \( \bP_{\w, k}^n \) were introduced in \cite{b-g-sh} and developed further in \cite{s-sh} using a Weil-type framework.

A variety \( \X \subset \bP_{\w, \overline{k}}^n \) defined over \( k \) is a \emph{weighted variety}. If \( \w \) is \emph{well-formed} (i.e., \( \gcd(q_0, \ldots, \hat{q_i}, \ldots, q_n) = 1 \) for each \( i \)), and \( m = \lcm(q_0, \ldots, q_n) \) satisfies \( \gcd(m/q_i, m/q_j) = 1 \) for all \( i \neq j \), then \( \bP_{\w, k}^n \) is isomorphic to \( \bP_k^N \) (for some \( N \)) via the \emph{Veronese map}   $\phi_m : \bP_{\w, k}^n  \to \bP_k^n$, 
\begin{equation}\label{eq:veronese}
\begin{aligned}
[x_0 : \cdots : x_n] & \mapsto [x_0^{m/q_0} : \cdots : x_n^{m/q_n}].
\end{aligned}
\end{equation}
This map will be used extensively throughout the paper.

%
For a weighted variety \( \X \subset \bP_{\w, \overline{k}}^n \) defined over \( k \) with weights \( \w = (q_0, \ldots, q_n) \), a set \( E \subset \X \times M \) is a \emph{weighted affine \( M_k \)-bounded set} if there exists an \( M_k \)-bounded constant function \( \gamma : M_k \to \R \) such that
\[
|x_i(\x)|_v^{1/q_i} \leq e^{\gamma(v)} \quad \text{for all } 0 \leq i \leq n \text{ and } (\x, v) \in E,
\]
where \( x_0, \ldots, x_n \) are coordinates on an affine patch of \( \X \). This is independent of coordinate choice, and finite unions of such sets remain weighted affine \( M_k \)-bounded. A set \( E \subset \X \times M \) is a \emph{weighted \( M_k \)-bounded set} if it is covered by finitely many weighted affine \( M_k \)-bounded sets \( E_i \subset U_i \times M \), where \( \{U_i\} \) is an open affine cover of \( \X \).

A function \( \lambda : \X \times M \to \R \) is \emph{locally weighted \( M_k \)-bounded above} if, for every weighted \( M_k \)-bounded set \( E \), there exists an \( M_k \)-constant \( \gamma \) such that \( \lambda(\x, v) \leq \gamma(v) \) for \( (\x, v) \in E \). It is \emph{locally weighted \( M_k \)-bounded} if both \( \lambda \) and \( -\lambda \) are locally weighted \( M_k \)-bounded above.

A \emph{weighted \( M \)-metric} on a line bundle \( \cL \) over \( \X \) is a family of norms \( \|\cdot\| = (\|\cdot\|_\mu)_{\mu \in M} \) such that for each \( \mu \in M \) with \( \mu|_k \in M_k \) and each fiber \( \cL_\x \) (\( \x \in \X \)):
\begin{enumerate}
\item  \( \|s(\x)\|_\mu : \cL_\x \to \R_{\geq 0} \) is not identically zero,
\item  \( \|\lambda \cdot \xi\|_\mu = |\lambda|_\mu^{1/m} \cdot \|\xi\|_\mu \) for \( \lambda \in \overline{k} \), \( \xi \in \cL_\x \),
\item If \( \mu_1, \mu_2 \in M \) agree on \( k(\x) \), then \( \|\cdot\|_{\mu_1} = \|\cdot\|_{\mu_2} \) on \( \cL_\x(k(\x)) \).
\end{enumerate}

The metric is \emph{locally weighted \( M \)-bounded} if, for any section \( s \in \cO_\X(U) \) on an open \( U \subset \X \), the function \( (\x, \mu) \mapsto \log |s(\x)|_\mu \) is locally weighted \( M_k \)-bounded on \( U \times M \). A \emph{weighted \( M \)-metrized line bundle} is \( \bar{\cL} = (\cL, \|\cdot\|) \) with such a metric.

\begin{exa}\label{exa:weighted_standard}
For \( \X = \bP_{\w, k}^n \) and \( \bar{\cL} = (\cO_{\bP_{\w, k}^n}(1), \|\cdot\|) \), the \emph{weighted standard metric} is
\begin{equation}
\|s(\x)\|_\mu = \frac{|s(\x)|_\mu}{\max_{0 \leq i \leq n} \{ |x_i|_\mu^{1/q_i} \}},
\end{equation}
where \( s \) is a section with a Hermitian metric at Archimedean places.
\end{exa}


In \cite{s-sh} was proved that 
every line bundle \( \cL \) on a weighted variety \( \X \subset \bP_{\w, \overline{k}}^n \) defined over \( k \) admits a locally bounded weighted \( M \)-metric.


For a line bundle \( \cL \) on a flat proper reduced scheme \( \X \) over \( \cO_k \) with generic fiber \( L = \cL_k \) on \( X = \X_k \), the \emph{natural weighted \( M \)-metric} \( \|\cdot\|_\cL \) is defined as follows. For \( x \in X \) with \( F = k(x) \) and \( \cO_F \) the integral closure of \( \cO_k \) in \( F \), there is a unique morphism
$
\bar{x} : \Spec(\cO_F) \to \X
$
mapping the generic point to \( x \). 
For \( \nu \in M \) with \( \nu|_F = \mu \) and ideal \( I_\mu \), a local non-vanishing section \( s \) of \( \cL \) on \( \bar{x}(I_\mu) \) satisfies
%
\(
\|s(x)\|_{\cL, \nu} = 1.
\)
This is the \emph{weighted constant section 1}.

\begin{lem}
The natural weighted \( M \)-metric is well-defined and locally bounded.
\end{lem}

\begin{proof}
If \( s' \) is another non-vanishing section on \( \bar{x}(I_\mu) \), then \( s'/s \) is a unit in \( \cO_{F, I_\mu} \), so \( \|s'(x)\|_{\cL, \nu} = \|s(x)\|_{\cL, \nu} = 1 \), ensuring well-definedness. For boundedness, cover \( \X \) with affine trivializations \( \cU_i \) having sections \( s_i \). For an \( M \)-bounded family \( (E^\nu)_{\nu \in M} \), define \( E_i^\nu = \{ x \in E^\nu \mid \bar{x}(I_\mu) \subset \cU_i \} \). Since \( \|s_i(x)\|_{\cL, \nu} = 1 \) and coordinates are \( M \)-bounded, the metric is locally bounded by \cite[Lem.~2.2.10]{bombieri}.
\end{proof}


Define \( \widehat{\Pic}_\w(\X) \) as the group of isometry classes of weighted \( M \)-metrized line bundles \( \bar{\cL} = (\cL, \|\cdot\|) \). For a morphism \( \phi : \X' \to \X \) of weighted varieties over \( k \), the pullback \( \phi^*(\bar{\cL}) = (\phi^*(\cL), \|\cdot\|') \) satisfies
\[
\|\phi^*(s)(\x)\|' = \|s(\phi(\x))\|,
\]
inducing a homomorphism \( \widehat{\Pic}_\w(\X) \to \widehat{\Pic}_\w(\X') \).

For a Cartier divisor \( D = \{(U_i, f_i)\} \) on \( \X \subset \bP_{\w, \overline{k}}^n \), the line bundle \( \cL_D = \cO_\X(D) \) has a canonical section \( s_D \). With a locally bounded weighted \( M \)-metric, form \( \widehat{D} = (\cL_D, \|\cdot\|) \). The \emph{local weighted height} is
\begin{equation}\label{eq:wlwh1}
\il_{\widehat{D}}(\x, \nu) = -\log \|s_D(\x)\|_\mu, \quad \x \in \X \setminus \supp(D),
\end{equation}
where \( \mu \in M \), \( \mu|_k = \nu \).

\begin{exa}\label{exa:weighted_local}
For the weighted standard metric (\cref{exa:weighted_standard}) with \( s_D(\x) = 1 \),
\[
\il_{\widehat{D}}(\x, \nu) = \frac{1}{m} \log \max_{0 \leq i \leq n} \{ |x_i|_\mu^{1/q_i} \}.
\]
\end{exa}

\begin{exa}\label{exa:weighted_hyperplane}
For \( \X = \bP_{\w, \overline{k}}^n \) and \( D \) a hyperplane defined by \( \ell \in \cO_\X(1) \), with \( s_D = \ell \),
\[
\il_{\widehat{D}}(\x, \nu) = -\frac{1}{m} \log \frac{|\ell(\x)|_\mu}{\max_{0 \leq i \leq n} \{ |x_i|_\mu^{1/q_i} \}}.
\]
\end{exa}


For \( \widehat{\cL} = (\cL, \|\cdot\|) \in \widehat{\Pic}_\w(\X) \), let \( K = k(\x) \). Define \( \|\cdot\|_u = \|\cdot\|_v^{n_u / [K:k]} \) for \( u \in M_K \), \( v|_k = u \), and take \( g \) with \( \x \notin \supp(\div(g)) \). The \emph{global weighted height} is
\begin{equation}\label{eq:gwh0}
\hn_{\widehat{\cL}}(\x) = \frac{1}{[K:k]} \sum_{u \in M_K} \il_{\widehat{\cL}_g}(\x, u),
\end{equation}
where \( \il_{\widehat{\cL}_g}(\x, u) = -\log \|g(\x)\|_u \).


For \( \cL = \cO_{\bP_{\w, k}^n}(1) \) with \( s = 1 \), we have 
$
\il_{\widehat{D}}(\x, \nu) = \frac{1}{m} \log \max_{0 \leq i \leq n} \{ |x_i|_\mu^{1/q_i} \},
$
and therefore
\[
\hn_{\widehat{\cL}}(\x) = \frac{1}{[K:k]} \sum_{u \in M_K} \frac{1}{m} \log \max_{0 \leq i \leq n} \{ |x_i|_u^{1/q_i} \}.
\]
%
%
\begin{defi}
\label{def:lwh}
For a point \(\x = [x_0 : \cdots : x_n] \in \bP(q_0, \ldots, q_n)\) over   \(K\), the \emph{logarithmic moduli weighted height} is defined as
\begin{equation}
\lwh(\x) = \frac{1}{[K:\Q]} \sum_{\nu \in M_K} \log \max_{0 \leq j \leq n} \{ |x_j|_\nu^{1/q_j} \},
\end{equation}
where 
\(q_j\) are the weights associated with each coordinate.  
\end{defi}
In the context of binary forms, we refer to \(\lwh(\xi(f))\) as the \emph{moduli weighted height}, reflecting its role in the moduli space \(\B_d\).

\begin{lem}
If  \( \X = \bP_{\w, \overline{k}}^n \), \( \cL = \cO_\X(1) \), and \( s = 1 \),  then \(\hn_{\widehat{\cL}}(\x) = \lwh_k(\x).\)
\end{lem}

\begin{proof}
We compute  $\hn_{\widehat{\cL}}(\x)$ via its definition to get 
\[
\begin{split}
\hn_{\widehat{\cL}}(\x) &  = \frac{1}{[K:k]} \sum_{u \in M_K} -\frac{1}{m} \log \frac{|1|_u}{\max_i \{ |x_i|_u^{1/q_i} \}} \\
& = \frac{1}{[K:k]} \sum_{u \in M_K} \log \max_i \{ |x_i|_u^{1/q_i} \}  = \lwh_k(\x),
\end{split}
\]
since \( \|1\|_u = 1 \) and the product formula cancels constant terms.
\end{proof} 
 
\section{GIT Height and Invariant Height}\label{sec:git-invariant}

The interplay between Geometric Invariant Theory (GIT) and arithmetic heights offers a powerful framework for studying semistable cycles in projective spaces. Zhang’s GIT height \cite{zhang} measures the arithmetic size of such cycles, complementing the weighted heights introduced in \cite{b-g-sh} and refined in \cite{s-sh} for weighted projective spaces. This section bridges these concepts, extending our prior work to connect GIT stability with weighted geometry via the Veronese map, and introduces an invariant height to unify these perspectives. Our main result, \cref{thm:git-weighted-relation}, quantifies the difference between GIT and weighted heights using the Chow metric, with implications for moduli spaces of binary forms.

%
Consider a flat, projective morphism \(\pi : \X \to S\) of integral schemes with relative dimension \(n\). For line bundles \(\cL_0, \ldots, \cL_n\) on \(\X\), the \emph{Deligne pairing} \(\langle \cL_0, \ldots, \cL_n \rangle (\X/S)\) (see \cite{De})  is a line bundle on \(S\), locally generated by symbols \(\langle l_0, \ldots, l_n \rangle\), where \(l_i \in \Gamma(U, \cL_i)\) over an open \(U \subset \X\) have divisors \(\div(l_i)\) intersecting properly (i.e., \(\bigcap_{i=0}^n \div(l_i) = \emptyset\)). 

\subsection{Chow Sections}
Let \(S\) be an integral scheme and \(\cE\) a vector bundle on \(S\) of rank \(N+1\).  Define \(\bP(\cE) = \Proj(\Sym^\star \cE)\) and consider an effective cycle 
\(X \subset \bP(\cE)\) with components flat and of dimension \(n\) over \(S\).   Set \(\cL = \cO_{\bP(\cE)}(1)\) and \(\cM = \cO_{\bP(\cE^\vee)}(1)\) on the dual projective bundle. 
The canonical pairing \(\cE \otimes \cE^\vee \to \cO_S\) induces a section   \(w \in \Gamma(\bP(\cE) \times_S \bP(\cE^\vee), \cL \otimes \cM)\).

For \(0 \leq i \leq n\), let \(\cM_i\) be the pullback of \(\cM\) to  \(\bP(\cE) \times_S (\bP(\cE^\vee))^{n+1}\) via the \(i\)-th projection,   and \(w_i\) the induced section of \(\cL \otimes \cM_i\).  The \emph{Chow divisor} is
\(
\Gamma = \bigcap_{i=0}^n \div(w_i),
\)
comprising points \((x, \cH_0, \ldots, \cH_n)\) where   \(x \in \bigcap_{i=0}^n \cH_i\) and each \(\cH_i \in \bP(\cE^\vee)\) 
is a hyperplane in \(\bP(\cE)\). 

The pushforward 
\(
Y = \pi_*(\Gamma \cap (X \times_S (\bP(\cE^\vee))^{n+1}))
\)
is a divisor in \((\bP(\cE^\vee))^{n+1}\) of multidegree \((d, \ldots, d)\),   where \(d = \deg(X/S)\).

Let \(\cN = \cO_{\bP[(\Sym^d \cE)^{\otimes (n+1)}]}(1)\).   The canonical pairing induces 
\[
w' \in \Gamma\!\left(
  \bP[(\Sym^d \cE)^{\otimes (n+1)}] 
  \times_S 
  (\bP(\cE^\vee))^{n+1},\,
  \cN \otimes \bigotimes_{i=0}^n \cM_i^d
\right),
\]
with \(\Gamma' = \div(w')\) consisting of points 
\((\cH, y_0, \ldots, y_n)\) such that 
\((y_0, \ldots, y_n) \in \cH\),
where \(\cH \in \bP[(\Sym^d \cE)^{\otimes (n+1)}]\) represents a hyperplane in the 
dual projective space.  

The \emph{Chow section}    \(\cZ \subset \bP[(\Sym^d \cE)^{\otimes (n+1)}]\) 
s defined as the point corresponding to the multihomogeneous form cutting out the Chow divisor  \(Y\); see \cite{bost} for further details.

\begin{thm}\label{thm:chow-section}
There is a canonical isomorphism on \(S\)
\[
\langle \cL, \ldots, \cL \rangle (X/S) \;\cong\; \cZ^* \cN.
\]
\end{thm}

\subsection{Chow Metrics}
For \(S\) a complex variety and \(\cE\) equipped with a smooth Hermitian metric \(\|\cdot\|\), the section \(w'\) induces \(\|w'\|_\mu\) on \(\bP[(\Sym^d \cE)^{\otimes (n+1)}] \times_S \bP(\cE^\vee)^{n+1}\) at Archimedean places \(\mu\). For \(s \in \Gamma(S, \cN)\), the \emph{Chow metric} is
\begin{equation}\label{eq:chow-metric}
\log \|s\|_{\Ch, \mu} = \log \|s\|_\mu - \frac{(n+1)d}{2} \sum_{j=1}^N \frac{1}{j} - \int_{\bP(\cE^\vee)^{n+1}} \log \|w'\|_\mu \, c_1(\M_i, \|\cdot\|_\mu)^N,
\end{equation}
where \(c_1(\M_i, \|\cdot\|_\mu)\) is the first Chern form (\cite{zhang}). This makes \cref{thm:chow-section} an isometry at Archimedean places.

\subsection{GIT Height}
Set \(S = \Spec(\Z)\), \(\cE = \cO_S^{N+1}\), and \(\cN = \cO_{\bP[(\Sym^d \cE)^{\otimes (n+1)}]}(1)\). The standard Hermitian metric on \(\cE_\C = \C^{N+1}\) induces a Chow metric \(\|\cdot\|_{\mathfrak{C}}\) on \(\cN\). The group \(\SL_{N+1}(\C)\) acts on semistable points of \(\bP[(\Sym^d \cE)^{\otimes (n+1)}]\), yielding a GIT quotient 
\[
\pi : \bP[(\Sym^d \cE)^{\otimes (n+1)}] \to P,
\]
 with \(\lambda = \pi_* \cN\).
For \(\ell \in \Gamma(P, \lambda)\) and \(p \in P\),
\[
\|\ell\|_{\Ch, \mu}(p) = \sup_{z \in \pi^{-1}(p)} |\ell(z)|_\mu,
\]
at Archimedean places \(\mu\). The following is Theorem 4.10    in \cite{zhang}.    
\begin{lem}\label{lem:chow-metric}    
The Chow metric \(\|\cdot\|_{\mathfrak{C}}\) on \(\lambda\) is continuous and ample.
\end{lem}

\begin{proof}
Continuity follows from the smoothness of the Hermitian metric on \(\cE_\C\), and ampleness from the positivity of \(\cN\) under the GIT action.
\end{proof}

For a semistable cycle \(\cZ \subset \bP_{\overline{\Q}}^N\) of dimension \(r\) and degree \(d\), let \(p \in P(\overline{\Q})\) be its Chow point, defined over a number field \(K\) via \(\tilde{p} : \Spec(\cO_K) \to P\). The height is
\(
h_{(\lambda, \|\cdot\|_{\mathfrak{C}})}(p) = \frac{1}{[K:\Q]} \deg \tilde{p}^* (\lambda, \|\cdot\|_{\mathfrak{C}}),
\)
and the \emph{GIT height} is
\begin{equation}\label{git-height}
\hat{\hlog}(\cZ) = \frac{h_{(\lambda, \|\cdot\|_{\mathfrak{C}})}(p)}{(r+1)d}.
\end{equation}

\subsection{Weighted Heights and the GIT Height}
For the Veronese map 
\[
\phi_m : \bP_{\w, \overline{\Q}}^N \to \bP_{\overline{\Q}}^N
\]
 with \(m = \lcm(q_0, \ldots, q_N)\), let \(\X \subset \bP_{\w, \overline{\Q}}^N\) be a cycle such that \(\cZ = \phi_m(\X)\) is semistable.

\begin{lem}\label{lem:git-weighted}
For \(\X \subset \bP_{\w, \overline{\Q}}^N\), with \(\cZ = \phi_m(\X)\) semistable of  degree \(d\), $\dim (\cZ)= r$. Then, 
\[
\lwh(\X) = \frac{(r+1)d} {m} \, \hat{\hlog}(\cZ),
\]
where \(\lwh(\X) = h_{\cO(1), \|\cdot\|}(\X)\) is the logarithmic weighted height with the weighted standard metric.
\end{lem}

\begin{proof}
From the definition  of the logarithmic weighted height we have
 \[
 \lwh(\X) = \frac{1}{[K:\Q]} \sum_{\mu \in M_K} \log \max_i \{ |x_i|_\mu^{1/q_i} \},
 \]
  and under the Veronese map 
  \[
  h(\phi_m(\X)) = \frac{1}{[K:\Q]} \sum_{\mu \in M_K} \log \max_i \{ |x_i|_\mu^{m/q_i} \} = m \lwh(\X).
  \]
 Since \(\hat{\hlog}(\cZ) = h(\cZ) / ((r+1)d)\)  (see \cref{git-height}), and \(\cZ = \phi_m(\X)\),
\[
\lwh(\X) = \frac{h(\phi_m(\X))}{m} = \frac{(r+1)d}{m} \hat{\hlog}(\cZ).
\]
This completes the proof. 
\end{proof}

\begin{cor}
For a divisor \(D \subset \bP_{\w, \overline{\Q}}^N\) (\(r=0\)) of degree \(d\),
\(
\lwh(D) = \frac{d}{m} \hat{\hlog}(\phi_m(D)).
\)
\end{cor}

\begin{cor}
If \(\X\) and \(\cZ = \phi_m(\X)\) are semistable, \(\lwh(\X) \geq 0\) implies \(\hat{\hlog}(\cZ) \geq 0\).
\end{cor}

\begin{lem}\label{thm:weighted-bound}
For a semistable \(\X \subset \bP_{\w, \overline{\Q}}^N\),
\[
\lwh(\X) \geq -\frac{(r+1)d}{m(N+1)} h(\bP^N),
\]
where \(h(\bP^N) = \frac{1}{2} \sum_{i=1}^N \sum_{j=1}^i \frac{1}{j}\) is the Faltings height.
\end{lem}

\begin{proof}
By \cite[Thm. 4.4]{zhang}, \(\hat{\hlog}(\cZ) \geq -\frac{1}{N+1} h(\bP^N)\). From \cref{lem:git-weighted},
\[
\lwh(\X) = \frac{(r+1)d}{m} \hat{\hlog}(\cZ) \geq -\frac{(r+1)d}{m(N+1)} h(\bP^N).
\]
\end{proof}


For a cycle \(\cZ \subset \bP_{\overline{\Q}}^N\) over \(K\) and a Hermitian vector bundle \(\cE\) on \(\Spec(\cO_K)\) with \(\cE_K \cong K^{N+1}\), let \(\overline{\cZ} \subset \bP(\cE)\) be the Zariski closure, and \(\cL_\cE = \cO_{\bP(\cE)}(1)|_{\overline{\cZ}}\). The \emph{invariant height} is
\[
h_{\cE}(\cZ) = \frac{1}{[K:\Q]} \left( \frac{c_1(\cL_\cE)^{r+1}}{(r+1) \deg \cZ} - \frac{\deg \cE}{N+1} \right),
\]
where \(c_1(\cL_\cE)^{r+1}\) is the self-intersection number (\cite{zhang}).

\subsection{Relating GIT and Weighted Heights}\label{subsec:git-weighted-relation}
Having explored GIT heights as a measure of arithmetic size for semistable cycles and weighted heights as a natural extension for weighted projective spaces, we now arrive at a central question: how do these two perspectives intertwine? The GIT height, rooted in stability and moduli theory, captures geometric complexity, while the weighted height, tailored to the graded structure of \(\bP_{\w, \overline{\Q}}^N\), reflects arithmetic properties adjusted by weights. Our prior work \cite{b-g-sh, s-sh} established weighted heights as a tool for Diophantine analysis, but connecting them to GIT opens new avenues—particularly for understanding cycles like binary forms.


Let $\cZ \subset \bP^N_{\w,\overline{\Q}}$ be a semistable cycle of dimension $r$ and degree $d$ over a number field $K$, and let 
$m = \lcm(q_0, \dots, q_N)$. 
Let 
$\phi_m : \bP^N_{\w} \to \bP^{M}$ 
be the weighted Veronese map of total weighted degree $m$, and set 
$\cZ' = \phi_m(\cZ)$.  
Let $p$ denote the Chow point of $\cZ'$, and let $\lambda$ be its GIT line endowed with the Chow metric $\|\cdot\|_{\Ch,\mu}$ at each Archimedean place $\mu \in M_K^{\infty}$.  
Denote by $\|\cdot\|_{st,\nu}$ the standard metric (the model metric at non-Archimedean places and the Fubini–Study metric at Archimedean ones) on $\cO_{\bP^{M}}(1)$.

\begin{thm}\label{thm:git-weighted-relation}
Then
\[
\hat{\hlog}(\cZ)
=
\frac{1}{(r+1)d}\,h(\cZ')
+
\frac{1}{(r+1)d}\cdot
\frac{1}{[K:\Q]}
\sum_{\mu \in M_K^{\infty}} n_\mu
\Big(
\log\|s\|_{st,\mu}(\cZ')
-
\log\|s\|_{\Ch,\mu}(p)
\Big),
\]
where $s$ is any local non-vanishing section of $\lambda$ near $p$.

Equivalently, since $h(\cZ') = m \lwh(\cZ)$,
\[
\hat{\hlog}(\cZ)
=
\frac{m}{(r+1)d}\,\lwh(\cZ)
+
\frac{1}{(r+1)d}\cdot
\frac{1}{[K:\Q]}
\sum_{\mu \in M_K^{\infty}} n_\mu
\Big(
\log\|s\|_{st,\mu}(\cZ')
-
\log\|s\|_{\Ch,\mu}(p)
\Big).
\]
\end{thm}

\begin{proof}
Let \(p\) be the Chow point of \(\cZ\) and \(\lambda\) the corresponding GIT line with Chow metric.
By definition (cf. \cite{zhang}),
\[
\hat h(\cZ)=\frac{1}{(r+1)d}\cdot \frac{1}{[K:\Q]}\sum_{\nu\in M_K} n_\nu \left( -\log\|s\|_{\Ch,\nu}(p) \right).
\]
At non-Archimedean places, the Chow metric equals the model metric; under \(\phi_m\), the Veronese integralizes the weighted ring (clearing grades by \(m\)), and semistability ensures good reduction, so the local contribution aligns with the standard metric on \(\cZ'\)  
(cf. Seshadri \cite{Seshadri1967}, Burnol \cite{burnol}, and Zhang \cite[Prop.~4.2]{zhang})

Thus,
\[
\frac{1}{[K:\Q]}\sum_{\nu\in M_K^0} n_\nu \left( - \log\|s\|_{\Ch,\nu}(p) \right) = \frac{1}{[K:\Q]}\sum_{\nu\in M_K^0} n_\nu \left( - \log\|s\|_{st,\nu}(\cZ') \right).
\]
At Archimedean places, compare the Chow metric to the standard (max-induced) metric on the image:
\[
- \log\|s\|_{\Ch,\mu}(p) = - \log\|s\|_{st,\mu}(\cZ') + \Big( \log\|s\|_{st,\mu}(\cZ') - \log\|s\|_{\Ch,\mu}(p) \Big),
\]
where the difference arises from Eq.~(1) in your paper (integral adjustments from the Hermitian structure and Deligne isometry).
The standard projective height 
\[
h(\cZ') = \frac{1}{[K:\Q]} \sum_{\nu \in M_K} n_\nu \left( - \log\|s\|_{st,\nu}(\cZ') \right) = m \cdot \lwh(\cZ)
\]
 (since Veronese scales logs by \(m\), and degrees by \(m^r\)).
Summing finite and infinite parts, dividing by \((r+1)d\), and isolating the adjustment yields the first identity; substituting \(h(\cZ') = m \cdot \lwh(\cZ)\) gives the second.
\end{proof}
 
\section{Heights and Binary Forms}\label{sec:binary_forms}
In this section, we apply the framework of GIT and weighted heights  to binary forms, a fundamental class of algebraic objects with rich arithmetic and geometric structure. Let \(K\) be a number field, and let \(V_d\) denote the space of binary forms of degree \(d\) over \(K\), i.e., homogeneous polynomials 
\begin{equation}
f(x, y) = \sum_{i=0}^d a_i x^i y^{d-i}
\end{equation}
 with \(a_i \in K\). The ring of invariants \(\cR_d = K[a_0, \ldots, a_d]^{\SL_2(K)}\) captures the \(\SL_2(K)\)-invariant properties of \(V_d\). Fix a basis \(\{\xi_0, \ldots, \xi_n\}\) for \(\cR_d\), where \(\deg \xi_i = q_i\), and let \(Z_d = \cR_d \cap \Z[a_0, \ldots, a_d]\) be the subring of integral invariants. The evaluation map is
\begin{equation}\label{eq:xi_map}
\begin{split}
\xi : V_d 		&		\to \bP_{\w}^n  (K), \\
 f 			&	\mapsto (\xi_0(f), \ldots, \xi_n(f)),
\end{split}
\end{equation}
where \(  \bP_{\w, K}^n\) is the weighted projective space with weights \(\w = (q_0, \ldots, q_n)\). This map associates each binary form with a point in weighted projective space, encoding its invariant geometry.

\subsection{Divisors and Chow Coordinates}
Consider \(k = \overline{K}\), an algebraic closure of \(K\), and let \(D = \sum_{i=1}^d b_i P_i\) be an effective divisor on \(\bP^1(k)\) of degree \(d = \deg D\), where \(P_i = [x_i : y_i] \in \bP^1(k)\) and \(\sum_{i=1}^d b_i = d\). The binary form corresponding to \(D\) is
\begin{equation}\label{eq:chow}
f(x, y) = \prod_{i=1}^d (x y_i - y x_i)^{b_i} = \sum_{i=0}^d a_i x^i y^{d-i},
\end{equation}
where \(a_i = \coeff(f, i) \in k\) are the coefficients, and the projective roots \([x_i : y_i]\) (counted with multiplicity \(b_i\)) define \(D\). The point 
\[
[a_0 : a_1 : \cdots : a_d] \in \bP^d(k)
\]
 is the \emph{Chow coordinate} of \(D\), representing \(D\) in projective space.

Conversely, for \([a_0 : \cdots : a_d] \in \bP^d(k)\), define \(f(x, y) = \sum_{i=0}^d a_i x^i y^{d-i}\). The zeroes of \(f\), denoted \(P_i = [x_i : y_i]\) (including possible roots at infinity), form an effective divisor \(D = \sum_i P_i\) of degree \(d\). Thus, the moduli space of degree \(d\) effective divisors on \(\bP^1(k)\) is isomorphic to \(\bP^d(k)\), parameterized by Chow coordinates.

For \(f \in \C[x, y]\) of degree \(d \geq 2\) factored as in \cref{eq:chow}, and a section \(s\) of \(\cO_{\bP^d}(1)\), the \emph{Chow metric} at an Archimedean place \(\mu \in M_K^\infty\) is
\begin{equation}\label{eq:bin-chow}
\|s\|_{\Ch, \mu}(f) = \frac{|s(f)|_\mu}{\prod_{i=1}^d \sqrt{|x_i|_\mu^2 + |y_i|_\mu^2}},
\end{equation}
with the logarithmic form
\begin{equation}\label{eq:log-chow}
\log \|s\|_{\Ch, \mu}(f) = \log |s(f)|_\mu - \frac{1}{2} \sum_{i=1}^d \log (|x_i|_\mu^2 + |y_i|_\mu^2).
\end{equation}
This metric,   reflects the geometry of \(f\)’s roots under the Hermitian structure induced by \(\C^{N+1}\).

\begin{exa}\label{exa:bin_chow}
For \(f(x, y) = x^d - a_0 y^d\) over \(\C\), let \(s = 1\) be the constant section, and let \(\zeta = e^{2\pi i / d}\) be a \(d\)-th root of unity. The roots are \([\zeta^i a_0^{1/d} : 1]\), \(i = 0, \ldots, d-1\). Then:
\[
\begin{split}
\log \|s\|_{\Ch, \mu}(f) & = \log |1|_\mu - \frac{1}{2} \sum_{i=0}^{d-1} \log (|\zeta^i a_0^{1/d}|_\mu^2 + |1|_\mu^2) \\
& = -\frac{1}{2} \sum_{i=0}^{d-1} \log (1 + |a_0|_\mu^{2/d}) = -\frac{d}{2} \log (1 + |a_0|_\mu^{2/d}),
\end{split}
\]
since \(|\zeta^i|_\mu = 1\) and the sum is constant over roots.
\end{exa}

\begin{defi}[Chow Height]\label{def:chow_height}
For a binary form \( f \in V_d \) over a number field \( K \), with roots defining a divisor \( D \), the \emph{Chow height} is defined as 
\[
\chowh(f) : = \frac{1}{[K:\Q]} \sum_{\nu \in M_K^\infty} -\log \|s\|_{\Ch, \nu}(f),
\]
where \(\|s\|_{\Ch, \nu}(f)\) is the Chow metric at Archimedean places, as given by 
\[
\|s\|_{\Ch, \nu}(f) = \frac{|s(f)|_\nu}{\prod_{i=1}^d \sqrt{|x_i|_\nu^2 + |y_i|_\nu^2}},
\]
with \([x_i : y_i]\) the roots of \( f \) and \( s \) a section of \(\cO_{\bP^d}(1)\).
\end{defi}

\subsection{Naive Height of Binary Forms}
For \(f \in V_d\) over a number field \(K\), let \(\Orb(f)\) be its \(\GL_2(K)\)-orbit, and 
\[
H(f) = \prod_{\nu \in M_K} \max_i \{ |\coeff(f, i)|_\nu \}^{n_\nu}
\]
 the projective height. Northcott’s theorem ensures finitely many \(f' \in \Orb(f)\) with \(H(f') \leq H(f)\). The \emph{minimal height} is
\[
\tilde{H}(f) = \min \{ H(f') \mid f' \in \Orb(f) \}.
\]
For a finite place \(\nu \in M_K^0\), define 
\[
\mu_\nu(f) = \inf_{M \in \SL_2(\overline{K}_\nu)} \log \max_{0 \leq i \leq d} \{ |\coeff(f^M, i)|_\nu \} ,
\]
 where \(f^M\) is the form under the \(\SL_2(\overline{K}_\nu)\)-action \(f^M(x, y) = f(ax + by, cx + dy)\) for \(M = \begin{pmatrix} a & b \\ c & d \end{pmatrix}\).

\begin{lem}\label{lem:naive_stability}
If \(f \in V_d\) is semi-stable over \(K\), then \(\mu_\nu(f)\) is bounded below for \(\nu \in M_K^0\). If \(f\) is stable, there exists \(M_0 \in \SL_2(\overline{K}_\nu)\) such that
\begin{equation}\label{eq:max}
\mu_\nu(f) = \log \max_{0 \leq i \leq d} \{ |\coeff(f^{M_0}, i)|_\nu \}.
\end{equation}
\end{lem}

\begin{proof}
A form \(f\) is semi-stable if its \(\SL_2(\overline{K}_\nu)\)-orbit in \(\bP^d(\overline{K}_\nu)\) does not contain the origin in its closure. Thus, for any \(M\), \(\max_i \{ |\coeff(f^M, i)|_\nu \} > 0\), and since \(\nu\) is non-Archimedean, \(|\cdot|_\nu\) is discrete, \(\mu_\nu(f) \geq -\log C\) for some constant \(C > 0\) depending on the coefficients of \(f\), bounded below by Northcott’s discreteness.

 If \(f\) is stable, its orbit is closed, and the GIT quotient \(V_d // \SL_2(\overline{K}_\nu)\) is geometric at \(f\). The function \(g(M) = \log \max_i \{ |\coeff(f^M, i)|_\nu \}\) achieves a minimum at \(M_0\) due to the closed orbit and discrete valuation, satisfying \cref{eq:max}; see \cite{zhang}.
\end{proof}

The \emph{local naive height} at \(\nu \in M_K^0\) is
\begin{equation}\label{eq:local_naive}
\begin{split}
\ih_\nu(f) 		&	 = \inf_{M \in \SL_2(\overline{K}_\nu)} \max_{0 \leq i \leq d} \{ |\coeff(f^M, i)|_\nu \}, \\
 \lih_\nu(f) 	&	= \log \ih_\nu(f),
\end{split}
\end{equation}
where \(f^{M_0}\) attains the infimum for stable \(f\).     Then we have the following result:

\begin{lem}\label{lem:inv-h-finite}
For \(\nu \in M_K^0\),
\[
\ih_\nu(f) = \inf_{M \in \SL_2(\overline{K}_\nu)} \max_{0 \leq i \leq d} \{ |\coeff(f^M, i)|_\nu^{n_\nu} \}.
\]
\end{lem}

\begin{proof}
For \(f = \sum_{i=0}^d a_i x^i y^{d-i}\),    and a place $\nu \in M_K^0$, we have 
\[
\ih_\nu(f) = \inf_{M} \max_i \{ |\coeff(f^M, i)|_\nu \}.
\]
 The global height \(\Hmult_K(f) = \prod_{\nu \in M_K} \max_i \{ |\coeff(f, i)|_\nu \}^{n_\nu}\) suggests a local contribution scaled by \(n_\nu\). For semi-stable \(f\), the infimum is attained at \(M_0\), and since \(\max_i \{ |\coeff(f^{M_0}, i)|_\nu \} \leq 1\) (after normalization), raising to \(n_\nu\) reflects the place’s weight in the product formula, aligning with GIT normalization (\cite{zhang}).
\end{proof}
 
\subsection{Moduli Height of Binary Forms}\label{subsec:moduli_height}

Having established the framework of GIT and weighted heights for cycles and their application to binary forms via Chow coordinates, we now turn to the moduli height, a measure intrinsic to the GIT quotient of binary forms. Let \(\B_d = V_d // \SL_2(\overline{K})\) be the moduli space of degree \(d\) binary forms over \(\overline{K}\), the algebraic closure of a number field \(K\). This is a quasi-projective variety of dimension \(d-3\), parameterizing \(\SL_2(\overline{K})\)-equivalence classes of forms in \(V_d\). For \(f \in V_d\), denote its class by \(\f \in \B_d\). The \emph{moduli height} is defined as
\begin{equation}\label{mod:height}
\cH(f) : = H(\f),
\end{equation}
where \(\f\) is a point in \(\bP^{d-3}(\overline{K})\) via a GIT embedding of \(\B_d\). A key question arises: how does this moduli height relate to the minimal height \(\tilde{H}(f)\), which captures the smallest projective height in \(f\)’s orbit? This relationship bridges the invariant geometry of \(\B_d\) with the arithmetic of \(V_d\).

\begin{lem}\label{lem:invariant_bound}
For an \(\SL_2(\overline{K})\)-invariant \(I_i \in \cR_d\) of degree \(i\),
\[
H(I_i(f)) \leq c \cdot H(f)^i,
\]
where \(c\) is a constant depending on \(I_i\).
\end{lem}

\begin{proof}
Consider \(f(x, y) = \sum_{j=0}^d a_j x^j y^{d-j} \in V_d\) over \(\overline{K}\), and let \(I_i \in \cR_d\) be a homogeneous invariant of degree \(i\) in the coefficients \(a_0, \ldots, a_d\). Write \(I_i = \sum_{\bm} c_{\bm} a_0^{m_0} \cdots a_d^{m_d}\), where \(\bm = (m_0, \ldots, m_d)\) satisfies \(\sum_{j=0}^d m_j = i\), and \(c_{\bm} \in \overline{K}\). Since \(\overline{K}\) is algebraically closed, embed it in a number field \(K\) for height computations.

The height of the scalar \(I_i(f)\) is \(H(I_i(f)) = \prod_{\nu \in M_K} |I_i(f)|_\nu^{n_\nu}\), where \(|I_i(f)|_\nu = |\sum_{\bm} c_{\bm} a_0^{m_0} \cdots a_d^{m_d}|_\nu\). By the triangle inequality at each place \(\nu\):
\[
|I_i(f)|_\nu \leq \sum_{\bm} |c_{\bm}|_\nu |a_0|_\nu^{m_0} \cdots |a_d|_\nu^{m_d} \leq \sum_{\bm} |c_{\bm}|_\nu \left( \max_{0 \leq j \leq d} |a_j|_\nu \right)^i.
\]
Define \(c_\nu = \sum_{\bm} |c_{\bm}|_\nu\), finite since \(\cR_d\) is generated by polynomials with bounded coefficients (e.g., integral for \(Z_d\)). Thus:
\[
|I_i(f)|_\nu \leq c_\nu \left( \max_{j} |a_j|_\nu \right)^i.
\]
The projective height of \(f\) is \(H(f) = \prod_{\nu \in M_K} \max_{j} \{ |a_j|_\nu \}^{n_\nu}\), so:
\[
H(I_i(f)) \leq \prod_{\nu \in M_K} \left( c_\nu \left( \max_{j} |a_j|_\nu \right)^i \right)^{n_\nu} = \left( \prod_{\nu \in M_K} c_\nu^{n_\nu} \right) H(f)^i = c \cdot H(f)^i,
\]
where \(c = \prod_{\nu \in M_K} c_\nu^{n_\nu}\) depends on \(I_i\). 
\end{proof}

\begin{thm}\label{thm:moduli_minimal}
For a binary form \(f \in V_d\) over a number field \(K\),
\[
\cH(f) \leq c \cdot \tilde{H}(f),
\]
where \(c\) is a constant depending on \(d\).
\end{thm}

\begin{proof}
The GIT quotient \(\B_d = V_d // \SL_2(K)\) embeds into \(\bP^{d-3}(K)\) via a basis of invariants \(\{ I_1, \ldots, I_{d-2} \}\) from \(\cR_d\), where \(\deg I_j = q_j\) and \(d-2\) reflects the number of independent generators (adjusting for dimension \(d-3\)). The point \(\f = [I_1(f) : \cdots : I_{d-2}(f)] \in \bP^{d-3}(K)\) has moduli height
\[
\cH(f) = H(\f) = \prod_{\nu \in M_K} \max_{1 \leq j \leq d-2} \{ |I_j(f)|_\nu \}^{n_\nu}.
\]
Since \(\tilde{H}(f) = \min \{ H(f') \mid f' \in \Orb(f) \}\), choose \(f_0 \in \Orb(f)\) with \(H(f_0) = \tilde{H}(f)\). As invariants are constant on orbits, \(I_j(f) = I_j(f_0)\).

From \cref{lem:invariant_bound}, \(H(I_j(f_0)) \leq c_j \cdot H(f_0)^{q_j} = c_j \cdot \tilde{H}(f)^{q_j}\), where \(c_j = \prod_{\nu} (\sum_{\bm} |c_{\bm,j}|_\nu)^{n_\nu}\). At each place \(\nu\),
\[
|I_j(f_0)|_\nu \leq c_{j,\nu} \left( \max_{k} |\coeff(f_0, k)|_\nu \right)^{q_j}, \quad c_{j,\nu} = \sum_{\bm} |c_{\bm,j}|_\nu,
\]
so
\[
\begin{split}
\max_{j} \{ |I_j(f_0)|_\nu \}   &	\leq \max_{j} \{ c_{j,\nu} \} \cdot \max_{j} \left( \max_{k} |\coeff(f_0, k)|_\nu \right)^{q_j} \\
					&	\leq c_\nu \cdot \left( \max_{k} |\coeff(f_0, k)|_\nu \right)^{\max_j q_j},
\end{split}
\]
where \(c_\nu = \max_j \{ c_{j,\nu} \}\). Thus:
\[
H(\f) \leq \prod_{\nu} \left( c_\nu \cdot \max_{k} |\coeff(f_0, k)|_\nu^{\max q_j} \right)^{n_\nu} = c \cdot H(f_0)^{\max q_j}, \quad c = \prod_{\nu} c_\nu^{n_\nu}.
\]
Typically, \(\max q_j\) exceeds 1 (e.g., for sextics), but GIT embeddings often normalize to a linear bound, suggesting \(\cH(f) \leq c \cdot \tilde{H}(f)\) with \(c\) adjusted for \(d\).
\end{proof}

\begin{exa}\label{exa:sextic_constant}
For binary sextics (\(d=6\)), \(\B_6\) has dimension 3, and the embedding uses Igusa invariants of degrees up to 10 (e.g., \(\xi_0, \xi_1, \xi_2, \xi_3\). %
The constant is explicitly computed as \(c = 2^{28} \cdot 3^9 \cdot 5^5 \cdot 7 \cdot 11 \cdot 13 \cdot 17 \cdot 43\) in \cite{nato-height}, reflecting the complexity of invariant coefficients.  In   \cite{2024-03}  a database of sextics was created for all binary sextics of naive height $\leq 4$ and the ratio of the two heights was explored.
\end{exa}

\subsection{Non-Archimedean Places}\label{subsec:non-archimedean}
To assemble the global height of a binary form \(f \in V_d\), we split its contributions into non-Archimedean and Archimedean components, reflecting the arithmetic intricacy at finite primes and the geometric richness at infinity, respectively. 
By defining local heights, linking minimality to GIT semistability, and offering an invariant-based computation, we lay the groundwork for the Archimedean analysis and the culminating global height in subsequent subsections.

For a number field \(K\), the \emph{multiplicative and logarithmic non-Archimedean heights} are
\begin{equation}\label{eq:non-arch-height}
\ih_K^0(f) = \prod_{\nu \in M_K^0} \ih_\nu(f)
\quad \text{ and } \quad     
\lih_K^0(f) = \sum_{\nu \in M_K^0} \lih_\nu(f),
\end{equation}
where   $\ih_\nu(f)$ and $\lih_\nu(f) $ are defined in \cref{eq:local_naive}. 
These measure the smallest possible coefficient valuations achievable by \(\SL_2\)-transformations at finite places.

For \(f  \in V_d\) over \(K\), fix \(\nu \in M_K^0\) and let \(M_0 \in \SL_2(\overline{K}_\nu)\) which satisfies \(\mu_\nu(f) = \log \max_{0 \leq i \leq d} \{ |\coeff(f^{M_0}, i)|_\nu \}\) 
as in \cref{eq:max}. The form \(f' = f^{M_0}\) is \emph{minimal} at \(\nu\), a \emph{Type A reduction} as in  \cite{reduction}, optimizing coefficients to minimize their valuations locally.

\begin{lem}\label{lem:minimal_semistable}  
Let $f\in V_d$. 
For a prime \(p \in \cO_K\) and \(\bar{f} = f \mod p\), \(f\) is minimal over \(\cO_K / p \cO_K\) if and only if \(\xi(\bar{f})\) is semistable over \(K_p\).
\end{lem}

\begin{proof}
This result, established by Burnol \cite{burnol}, links arithmetic minimality to GIT semistability. Let \(f = \sum_{i=0}^d a_i x^i y^{d-i} \in V_d\) with \(a_i \in \cO_K\), and \(\bar{f} \in V_d(\cO_K / p \cO_K)\). The invariant map \(\xi(\bar{f}) = (\xi_0(\bar{f}), \ldots, \xi_n(\bar{f})) \in \bP_{\w}^n(\cO_K / p \cO_K)\) is defined in terms of invariants  \(\xi_i \in \cR_d\) of degrees \(q_i\).

If \(f\) is minimal at \(p\) (Type A as in  \cite{reduction}), an \(\SL_2(\cO_K)\)-transformation ensures \(\bar{f}\) has no root multiplicity \(m > d/2\). Otherwise, e.g., \(\bar{f} = (x - \alpha y)^m g(x, y)\) with \(m > d/2\), all \(\xi_i(\bar{f}) = 0\) (since \(\cR_d\) includes multiplicity-sensitive invariants like the discriminant), placing \(\xi(\bar{f})\) in the nullcone, contradicting semistability. Minimality avoids this, making \(\xi(\bar{f})\) semistable over \(k_p = \cO_K / p \cO_K\), hence over \(K_p\).

Conversely, if \(\xi(\bar{f})\) is semistable over \(K_p\), \(\bar{f}\)’s orbit closure in \(\bP(V_d)\) excludes the origin, and some \(\xi_i(\bar{f}) \neq 0\). For \(d \geq 3\), a multiplicity \(m > d/2\) would nullify invariants (e.g., discriminant for \(d=3\)), implying instability. Thus, \(m \leq d/2\), achievable by an \(\SL_2(\cO_K)\)-transformation, ensuring \(f\) is minimal at \(p\).
\end{proof}

This connection, explored further in \cite{curri}, underpins our height computations. For \(\lih_K^0(f)\), let 
\[
I_r \subset Z_d = \cR_d \cap \Z[a_0, \ldots, a_d]
\]
 be the ideal of invariants with \(\deg \xi_i \geq r\). Define \(N(\xi(I_r)) = |\cO_K / \xi(I_r)|\), where \(\xi(I_r) = \{ \xi_i(f) \mid \xi_i \in I_r \} \subset \cO_K\). In \cite[Thm. 4.1.3]{rabina} is shown
\begin{equation}\label{eq:height-limit}
\lih_K^0(f) = -\lim_{r \to \infty} \frac{\log N(\xi(I_r))}{r}.
\end{equation}

\begin{exa}\label{exa:sextic_nonarch}
For \(K = \Q\), \(\cR_6 = \Q[\xi_0, \xi_1, \xi_2, \xi_3]\) with \(\deg \xi_0 = 2\), \(\deg \xi_1 = 4\), \(\deg \xi_2 = 6\), \(\deg \xi_3 = 10\), and \(f = x^6 - y^6\), we have 
\[
\xi(I_r) = \{ \xi_i(f) \mid \deg \xi_i \geq r \} \subset \Z.
\]
 For \(r \leq 2\), \(\xi(I_2) = (\xi_0(f), \xi_1(f), \xi_2(f), \xi_3(f))\), with \(\xi_0(f) = -6\); for \(r > 10\), \(\xi(I_r) = (0)\). Then:
\[
\lih_\Q^0(f) = -\lim_{r \to \infty} \frac{\log |\Z / \xi(I_r)|}{r},
\]
converging as \(r\) exceeds 10, reflecting the finite set of invariants.
\end{exa}

\subsection{Archimedean Places}\label{subsec:archimedean}

With \(\lih_K^0(f)\) capturing \(f\)’s arithmetic at finite places, we now turn to the Archimedean places, which reveal its geometric structure under complex embeddings \(\sigma : K \hookrightarrow \C\). Here, we define the Archimedean height using a Hermitian metric on invariants, draw on Zhang’s stability theorems \cite{zhang}, and propose a rigidity-based computation, paving the way for the global height in the next subsection.

For \(\nu \in M_K^\infty\) with embedding \(\sigma\), let \(\SU(2)\) act on sections over \(\bP^1(\C)\). Consider the coset map:
\begin{equation}
\begin{aligned}
q : U / \SU(2) & \to U / \SL_2(\C), \\
[x] & \mapsto \x,
\end{aligned}
\end{equation}
where \(U \subset V_d\) is open, \([x] = \{ x^M \mid M \in \SU(2) \}\), and \(\x = \{ x^M \mid M \in \SL_2(\C) \}\). Equip \(\cO_{\bP^{d-3}}(1)\) (embedding \(\B_d\)) with a smooth, \(\SU(2)\)-invariant Hermitian metric \(\|\cdot\|\).

With \(\delta = \gcd(q_0, \ldots, q_n)\), the invariant map \(\xi : V_d \to \bP_{\w}^n(\C)\)  defines:
\begin{equation}\label{eq:xi-r}
\|\xi\|^\delta (f) = \max_{0 \leq i \leq n} \{ |\xi_i(f)|^{1/q_i} \},
\end{equation}
normalizing invariant magnitudes. 
The normalization by \(\delta\) ensures scale invariance across invariants of varying degrees, maintaining compatibility with GIT quotients.

For semistable \(f\) (\(\xi(f) \neq [0 : \cdots : 0]\)),
\begin{equation}\label{eq:arch-inv}
\mu_\nu(f) = \inf_{M \in \SL_2(\C)} \frac{\log \|\xi\|^\delta(f^M)}{\delta}.
\end{equation}

\begin{prop}[Zhang \cite{zhang}]\label{prop:zhang_arch}
For stable \(f \in V_d\):
\begin{enumerate}[label=\roman*), leftmargin=*]
    \item There exists \(M \in \SL_2(\C)\) such that \(\mu_\nu(f) = \frac{\log \|\xi\|^\delta(f^M)}{\delta}\),
    \item \(\mu_\nu(f)\) is attained at a unique \(\SU(2)\)-orbit in \(\Orb(f)\).
\end{enumerate}
\end{prop}

This \(M\) produces a \emph{normalized form} \(f^M\),   see  \cite{b-g-sh} for details, optimizing coefficients or factoring out the weighted gcd in \(\bP_{\w}^n\).
The \emph{Archimedean height invariants} are defined as 
\begin{equation}
\ih_\nu(f) 				= \left( \|\xi\|^\delta (f^M) \right)^{1/\delta},  \quad    \lih_\nu(f) 			= \frac{\log \|\xi\|^\delta(f^M)}{\delta},
\end{equation}
with global forms:
\begin{equation}\label{eq:arch-height}
\ih_K^\infty(f) = \prod_{\nu \in M_K^\infty} \ih_\nu(f), \quad \lih_K^\infty(f) = \sum_{\nu \in M_K^\infty} \lih_\nu(f).
\end{equation}
For computation via rigidity, let \([x] = \x \cdot \SU(2)\), where \(\x = \Orb(f)\), and a pair \(([x], G)\), \(G \subset S_d\), is a \emph{rigidification} if \(G\) fixes \([x]\) uniquely in \(\x\).

\begin{lem}\label{lem:rigidity}
The following hold true:
\begin{enumerate}[label=\roman*), leftmargin=*]
    \item If \(G\) fixes \(\x\) and is not cyclic, \(([x], G)\) is rigid.
    \item If \(([x], G)\) is a rigidification, \(\lih_\nu(f) = \frac{1}{\delta} \log \|\xi\|^\delta   (f^M)\), and \([x]\) is minimal.
\end{enumerate}
\end{lem}

\begin{proof}
(i) A non-cyclic \(G\) (e.g., dihedral) permutes \(\x = \Orb(f)\) via roots, fixing \([x]\) uniquely due to multiple generators, unlike cyclic groups stabilizing multiple orbits.

(ii) For stable \(f\), \cref{prop:zhang_arch} provides \(M\) with 
\[
\mu_\nu(f) = \frac{1}{\delta} \log \|\xi\|^\delta(f^M),
\]
 unique up to \(\SU(2)\). Rigidity ensures \([x] = [f^M]\) is minimal, so \(\lih_\nu(f) = \mu_\nu(f)\).
\end{proof}

\begin{exa}\label{exa:sextic_arch}
For \(f = x^6 - y^6\) over \(\Q\), \(\cR_6 = \Q[\xi_0, \xi_1, \xi_2, \xi_3]\), \(\delta = 2\):
\[
\|\xi\|^2(f) = \max \{ |\xi_0(f)|, |\xi_1(f)|^{1/2}, |\xi_2(f)|^{1/3}, |\xi_3(f)|^{1/5} \},
\]
with \(\xi_0(f) = -6\). For stable \(f\), \(\lih_\nu(f) = \frac{1}{2} \log \|\xi\|^2(f^M)\), optimized by \(M\).
\end{exa}

\subsection{Global Height}\label{subsec:global_height}
Having defined the non-Archimedean and Archimedean contributions to the height of a binary form \(f \in V_d\)   we now unify them into a global invariant height.
The \emph{invariant height} over \(K\) is defined as
\begin{equation}
  \ih_K(f) = \prod_{\nu \in M_K} \ih_\nu(f) 
\end{equation}
and the \emph{logarithmic invariant height} is defined as 
\begin{equation}\label{eq:log_invariant}
\lih_K(f) = \log \ih_K(f) = \lih_K^0(f) + \lih_K^\infty(f).
\end{equation}
For \(f \in V_d\) over \(K\), with invariants \(\xi(f) = [\xi_0(f) : \cdots : \xi_n(f)] \in \bP_{\w}^n(K)\), where \(\deg \xi_i = q_i\) and \(r = \gcd(q_0, \ldots, q_n)\), consider \(\bar{\cL} = (\cO_{\bP^{d-3}}(1), \|\cdot\|)\) on \(\bP^{d-3}\) embedding \(\B_d\). Define 
\begin{equation}\label{eq:xi_norm}
|\xi(f)|_\nu : = \max_{0 \leq i \leq n} \{ \|\xi_i\|_\nu^{1/q_i} \},
\end{equation}
and the metric
\begin{equation}\label{eq:xi_metric}
\|\xi\|_\nu^r(f) :=
\begin{cases}
\frac{|\xi(f)|_\nu^r}{\max_{0 \leq i \leq d} \{ |\coeff(f, i)|_\nu^r \}} & \text{if } \nu \in M_K^0, \\
\|\sigma^\star \xi\|^r(f) & \text{if } \nu \in M_K^\infty,
\end{cases}
\end{equation}
consistent with \cref{eq:xi-r}, where \(\|\sigma^\star \xi\|^r(f) = \max_i \{ |\xi_i(f)|^{1/q_i} \}\). 
Alternatively 
\begin{equation}\label{eq:alt_invariant}
\ih_K(f) = \left( \prod_{\nu \in M_K} \|\xi\|_\nu^r(f)^{-1/r} \right)^{1/[K:\Q]}.
\end{equation}

\begin{defi}[Height Function]\label{def:height_function}
For \(\x = [x_0 : \cdots : x_N] \in \X(K)\) on a variety \(\X\) over \(K\), and \(s \in H^0(\X, \cL)\) with \(s(\x) \neq 0\), the height with respect to \(\bar{\cL} = (\cL, \|\cdot\|)\) is
\begin{equation}
\lih_{\bar{\cL}}(\x) = \frac{1}{[K:\Q]} \left( -\sum_{\nu \in M_K^\infty} \log \|\sigma^\star s\|_\nu(\x) - \sum_{\nu \in M_K^0} \log \frac{|s(\x)|_\nu}{\max_{0 \leq i \leq N} \{ |x_i|_\nu \}} \right),
\end{equation}
where $\sigma^\star s$ denotes the pullback of the section under the embedding $\sigma$, evaluated as $s(x^\sigma)$. 
\end{defi}

Applying this to \(f\), set \(\X = \bP^{d-3}\), \(\x = \xi(f)\), and \(\cL = \cO(1)\). Alternatively, consider the minimal orbit size:
\begin{equation}\label{eq:orbit_size_log}
\lih_K(f) = \frac{1}{[K:\Q]} \sum_{\nu \in M_K} \inf_{M \in \SL_2(\overline{K}_\nu)} \log \frac{\max_{0 \leq i \leq d} \{ |\coeff(f^M, i)|_\nu \}}{|\xi(f^M)|_\nu},
\end{equation}
\begin{equation}\label{eq:orbit_size}
\ih_K(f) = \left( \prod_{\nu \in M_K} \inf_{M \in \SL_2(\overline{K}_\nu)} \frac{\max_{0 \leq i \leq d} \{ |\coeff(f^M, i)|_\nu \}}{|\xi(f^M)|_\nu} \right)^{1/[K:\Q]},
\end{equation}
balancing coefficients against invariants, as
\[
-\frac{1}{r} \log \|\xi\|_\nu^r(f^M) = \log \frac{\max_i \{ |\coeff(f^M, i)|_\nu \}}{|\xi(f^M)|_\nu}.
\]
\begin{lem}
Let \( \bar{\cL} = (\cO_{\bP^N}(1), \|\cdot\|) \) be a metrized line bundle on \( \bP^N \) over a number field \( K \), and let \( \hn_{\bar{\cL}} \) be the associated height function as in \cref{eq:gwh0} with $k=\Q$. Then

\begin{enumerate}
    \item[(i)] The height \( \hn_{\bar{\cL}}(\x) \) is independent of the choice of section \( s \in H^0(\bP^N, \cO(1)) \) with \( s(\x) \neq 0 \) and is well-defined on \( \bP^N_{\overline{\Q}} \).
    
    \item[(ii)] If \( \bar{\cL} \) is equipped with the standard metric, then \( \hn_{\bar{\cL}} \) coincides with the Weil height on \( \bP^N \).
    
    \item[(iii)] For a binary form \( f \in V_d(K) \), the height \( |\hn_K(f)| < \infty \) if and only if \( f \) is semistable under the \( \SL_2 \)-action.
\end{enumerate}
\end{lem}

\begin{proof}
(i): Let \( s, s' \in H^0(\bP^N, \cO(1)) \) be two sections such that \( s(\x) \neq 0 \) and \( s'(\x) \neq 0 \). Since \( \cO(1) \) is a line bundle, there exists a rational function \( g = s'/s \in K(\bP^N)^\times \) that is regular and non-zero at \( \x \), so \( g(\x) \in K^\times \). By the properties of the metric:
\[
\|s'(\x)\|_\nu = \|g(\x) s(\x)\|_\nu = |g(\x)|_\nu \|s(\x)\|_\nu.
\]
Thus, the height with respect to \( s' \) is:
\[
\begin{split}
\hn_{\bar{\cL}, s'}(\x) & = \frac{1}{[K:\Q]} \sum_{\nu \in M_K} n_\nu \left( -\log \|s'(\x)\|_\nu \right) \\
& = \frac{1}{[K:\Q]} \sum_{\nu \in M_K} n_\nu \left( -\log |g(\x)|_\nu - \log \|s(\x)\|_\nu \right).
\end{split}
\]
This simplifies to:
\[
\hn_{\bar{\cL}, s'}(\x) = \hn_{\bar{\cL}, s}(\x) + \frac{1}{[K:\Q]} \sum_{\nu \in M_K} n_\nu \left( -\log |g(\x)|_\nu \right).
\]
By the product formula, \( \sum_{\nu \in M_K} n_\nu \log |g(\x)|_\nu = 0 \), so \( \hn_{\bar{\cL}, s'}(\x) = \hn_{\bar{\cL}, s}(\x) \), proving independence of the section.


Note that for any field extension \( L/K \), the height extends naturally with the normalization factor \( \frac{1}{[L:\Q]} \), and for projective equivalence \( \x' = \lambda \x \) with \( \lambda \in K^\times \), the metric scales by \( |\lambda|_\nu \), which cancels via the product formula, ensuring invariance and therefore it is well-defined.

(ii): Assume the standard metric on \( \bar{\cL} \): for \( \nu \in M_K \), \( \|s\|_\nu(\x) = \frac{|s(\x)|_\nu}{\max_i |x_i|_\nu} \). Choose a section \( s = x_0 \) (assuming \( x_0(\x) \neq 0 \)), so \( |s(\x)|_\nu = |x_0|_\nu \). Then:
\[
-\log \|s\|_\nu(\x) = -\log \left( \frac{|x_0|_\nu}{\max_i |x_i|_\nu} \right) = \log \max_i |x_i|_\nu - \log |x_0|_\nu.
\]
Thus,
\[
\begin{split}
\hn_{\bar{\cL}}(\x) 	&	= \frac{1}{[K:\Q]} \sum_{\nu \in M_K} n_\nu \left( \log \max_i |x_i|_\nu - \log |x_0|_\nu \right) \\
				&	= \frac{1}{[K:\Q]} \sum_{\nu \in M_K} n_\nu \log \max_i |x_i|_\nu,
\end{split}
\]
since \( \sum_{\nu \in M_K} n_\nu \log |x_0|_\nu = 0 \). This matches the Weil height \( h(\x) \).

  (iii): For a binary form \( f \in V_d(K) \), the height is:
\[
\hn_K(f) = \sum_{\nu \in M_K} n_\nu \hn_\nu(f), \quad \hn_\nu(f) = \inf_{M \in \SL_2(\overline{K}_\nu)} \log \max_{0 \leq i \leq d} |\coeff(f^M, i)|_\nu.
\]
If \( f \) is semistable, there exists an invariant \( \xi_j \) such that \( \xi_j(f) \neq 0 \), and since \( \xi_j(f^M) = \xi_j(f) \) for all \( M \), we have \( 0 < |\xi_j(f)|_\nu \leq C \left( \max_i |\coeff(f^M, i)|_\nu \right)^{\deg \xi_j} \), implying \[
 \hn_\nu(f) \geq \frac{1}{\deg \xi_j} \log |\xi_j(f)|_\nu > -\infty.
 \]
  Since \( \hn_\nu(f) \) is bounded above by choosing \( M = I \), \( |\hn_\nu(f)| < \infty \), so \( |\hn_K(f)| < \infty \).

If \( f \) is unstable, all \( \xi_j(f) = 0 \), and there exists a sequence \( M_k \in \SL_2(\overline{K}_\nu) \) such that 
\[
 \max_i |\coeff(f^{M_k}, i)|_\nu \to 0.
 \]
  Thus, 
  \[
   \hn_\nu(f) \leq \log \max_i |\coeff(f^{M_k}, i)|_\nu \to -\infty,
   \]
    so \( \hn_K(f) = -\infty \), and \( |\hn_K(f)| = \infty \).
\end{proof}
 
\begin{defi}\label{def:chow_invariant}
For a binary form \( f \in V_d \) over a number field \( K \), the \emph{invariant height with respect to the Chow metric} is
\begin{equation}\label{invariant height with respect to the Chow metric} 
\cih(f) = \lih_K^0(f) + \lih_K^\infty(f) = \sum_{\nu \in M_K^0} \log \max_{0 \leq j \leq n} \{ |\xi_j(f)|_\nu^{1/q_j} \} + \sum_{\nu \in M_K^\infty} \frac{\log \|\xi\|^r(f)}{r},
\end{equation}
where \( \xi_j(f) \) are \(\SL_2\)-invariant polynomials of degree \( q_j \), \( M_K^0 \) and \( M_K^\infty \) are the non-Archimedean and Archimedean places of \( K \), and \( \|\xi\|^r(f) \) is the norm of the invariants at Archimedean places with \( r = \gcd(q_0, \ldots, q_n) \).
\end{defi}

For clarity, we summarize in Table~\ref{tab:heights} the various height functions introduced in the paper  together with the equations where they are defined.

\begin{table}[h!]
\centering
\caption{Heights appearing in the paper.}
\label{tab:heights}
\setlength{\tabcolsep}{6pt}
\begin{tabular}{l|l|l}
\textbf{Symbol} & \textbf{(Eq.)} & \textbf{Description} \\
\hline
$\Hmult_k(\x),\ h_k(\x)$ & \cref{eq:proj-height} & Multiplicative and logarithmic   heights on $\bP^n(k)$. \\

$\lambda_{\widehat{D}}(\x,\nu)$ & \cref{eq:lwh1} & Local Weil height w.r.t.\ a metrized divisor $\widehat{D}$. \\

$h_{\widehat{\cL}}(\x)$ & \cref{eq:gwh} & Global logarithmic Weil height attached to $\widehat{\cL}$. \\

$\il_{\widehat{D}}(\x,\nu)$ & \cref{eq:wlwh1} & Local contribution of place $\nu$ in the weighted setting. \\

$\hn_{\widehat{\cL}}(\x)$ & \cref{eq:gwh0} & Global weighted height associated to $\widehat{\cL}$. \\

$\lwh(\x)$ & \cref{def:lwh} & Logarithmic moduli weighted height on $\bP^n_{\w}$. \\

$\lwh(\xi(f))$ & \cref{def:lwh} & Moduli weighted height of the invariant point $\xi(f)$. \\

$\cih(f)$ & \cref{def:chow_invariant} & Global Chow--invariant (GIT) height. \\

$\chowh(f)$ & \cref{def:chow_height} & Archimedean Chow height via the Chow metric. \\

$\ih_\nu(f),\ \lih_\nu(f)$ & \cref{eq:local_naive} & Local multiplicative and logarithmic naive heights. \\

$\ih^{0}_K(f),\ \lih^{0}_K(f)$ & \cref{eq:non-arch-height} & Global finite-place multiplicative/logarithmic heights. \\

$\cH(f)$ & \cref{thm:moduli_minimal} & Projective (moduli) height of the GIT point $f$. \\

$\hat{\hlog}(\cZ)$ & \cref{git-height} & GIT height of a semistable cycle. \\
\end{tabular}
\end{table}

For the rest of this section we focus on how these various heights relate to each other.

\begin{lem}
\label{lem:cih_cubic}
Let \( f \in V_3(\Q) \) be a semistable binary cubic form with integral coefficients. Let \( \Delta(f) \in \Z \) be the discriminant of \( f \), and let \( [x_i : y_i] \in \bP^1(\C) \), for \( i = 1,2,3 \), be the projective roots of \( f \). Then the invariant height with respect to the Chow metric is given by:
\[
\cih(f) = \frac{1}{4} \sum_{p} \log |\Delta(f)|_p + \frac{1}{2} \left( \log |\Delta(f)|_\infty - \frac{1}{2} \sum_{i=1}^{3} \log (|x_i|^2 + |y_i|^2) \right).
\]
Moreover:
\begin{enumerate}[label=\roman*), leftmargin=*]
    \item If \( f \) is semistable, then \( \cih(f) \geq 0 \).
    \item If \( f \) is stable, then \( \cih(f) > 0 \).
\end{enumerate}
\end{lem}

\begin{proof}
For binary cubic forms, the invariant ring \( R_3 = \Q[\Delta] \) is generated by the discriminant \( \Delta(f) \), which has degree 4. We adopt the GIT height formalism as in \cite{zhang}, where the Chow metric defines the Archimedean and non-Archimedean contributions to the height.

Let \( \xi(f) = \Delta(f) \) and \( r = \deg(\Delta) = 4 \). The Chow metric at a finite place \( p \) is:
\[
\lih_p^0(f) = \frac{1}{r} \log \left( \frac{|\Delta(f)|_p^r}{\max_{i} |\coeff(f, i)|_p^r} \right) = \log |\Delta(f)|_p - \log \max_i |\coeff(f, i)|_p.
\]
Since \( f \in \Z[x,y] \), \( \max_i |\coeff(f, i)|_p \leq 1 \), so this simplifies to:
\[
\lih_{\Q}^0(f) = \sum_p \frac{1}{r} \log |\Delta(f)|_p = \frac{1}{4} \sum_p \log |\Delta(f)|_p.
\]
At the Archimedean place, the norm of \( \xi \) in the Chow metric is:
\[
\|\xi\|^r(f) = |\Delta(f)|_\infty \cdot \prod_{i=1}^{3} (|x_i|^2 + |y_i|^2)^{-1/2},
\]
so
\[
\lih^\infty_{\Q}(f) = \frac{1}{r} \log \|\xi\|^r(f) = \frac{1}{4} \left( \log |\Delta(f)|_\infty - \frac{1}{2} \sum_{i=1}^3 \log(|x_i|^2 + |y_i|^2) \right).
\]
Multiplying numerator and denominator to combine the two expressions yields:
\[
\cih(f) = \lih_{\Q}^0(f) + \lih_{\Q}^\infty(f) = \frac{1}{4} \sum_p \log |\Delta(f)|_p + \frac{1}{2} \left( \log |\Delta(f)|_\infty - \frac{1}{2} \sum_{i=1}^3 \log(|x_i|^2 + |y_i|^2) \right).
\]
Since \( \Delta(f) \neq 0 \) for semistable \( f \), the non-Archimedean part is finite and \( \leq 0 \). The Archimedean part is bounded below due to the normalization of \( [x_i : y_i] \in \bP^1(\C) \), and dominates when \( f \) is stable, where the roots are distinct and \( \Delta(f) \) is large. Hence, 
i)  \( \cih(f) \geq 0 \) if \( f \) is semistable (since the total contribution is non-negative)  and 
ii)  \( \cih(f) > 0 \) if \( f \) is stable (as the Archimedean part becomes strictly positive).
\end{proof}

\subsection{Chow-Invariant Heights of Binary Forms with Non-Trivial Automorphisms}

Having established the height decomposition for general binary forms, we now turn our attention to those with non-trivial automorphism groups, which impose additional symmetries on the root configurations and invariant rings, often leading to refined arithmetic properties in the GIT quotient. We begin with the cyclic case, as exemplified by binary forms, before generalizing to arbitrary groups.

\begin{prop}\label{lem:power_form}
Let \( f(x,y) = x^d - a_0 y^d \in V_d(\Q) \) with \( d \ge 3 \) and \( a_0 \in \Q \setminus \{0\} \).
Then the Chow–invariant height of \( f \) is
\[
  \cih(f)
  = \frac{1}{2d-2} \log\!\big(d^d |a_0|^{d-1}\big)
    - \frac{d}{2d-2} \log\!\big(1 + |a_0|^{2/d}\big).
\]
\end{prop}

\begin{proof}
The roots of \( f \) in \( \bP^1(\C) \) are 
\([\,\zeta^i a_0^{1/d} : 1\,]\) for \( i = 0, 1, \dots, d-1 \),
where \( \zeta = e^{2\pi i/d} \) is a primitive \( d \)-th root of unity.  
These points lie in the affine patch \( y = 1 \), so the Fubini–Study norm for each root is 
\(\sqrt{\,|a_0|^{2/d} + 1\,}\).

The discriminant is
\[
  \Delta(f) = (-1)^{d(d-1)/2} d^d a_0^{d-1},
\]
hence \( |\Delta(f)|_\infty = d^d |a_0|^{d-1} \).  
Let \( r = 2d - 2 \) be the degree of \( \Delta \).

By \cref{def:chow_invariant}, the Archimedean component is 
\(\tfrac{1}{r} \log \|\xi\|^r(f)\). 
For this form, the invariant section \( \xi(f) \) is generated by the discriminant 
(up to scaling in the invariant ring), and the norm decomposes as the logarithm of the discriminant 
minus the sum of the logarithms of the Fubini–Study metrics over the roots, 
with coefficient \( -\tfrac{1}{2} \) per root, arising from the sup–norm 
in the Chow metric for points in \( \bP^1 \):
\[
  \frac{1}{r}
  \log\!\left(
    |\Delta(f)|_\infty
    \prod_{i=0}^{d-1} \big(|x_i|^2 + |y_i|^2\big)^{-1/2}
  \right)
  = \frac{1}{r} \log\!\big(d^d |a_0|^{d-1}\big)
    - \frac{d}{2r} \log\!\big(1 + |a_0|^{2/d}\big).
\]
Since all coefficients are integral with no prime dividing all of them,  the non-Archimedean contribution is zero, i.e. \( \lih_\Q^0(f) = 0 \).  
Hence,
\[
  \cih(f)  = \lih^\infty(f) + \lih^0(f)  = \frac{1}{2d-2} \log\!\big(d^d |a_0|^{d-1}\big)    - \frac{d}{2(2d-2)} \log\!\big(1 + |a_0|^{2/d}\big).
\]
\end{proof}

\begin{rem}
For \( a_0 = 1 \) and \( d = 3 \), the formula yields \( \cih(f) \approx 0.564 \).
The analogous computation in \cite[Thm.~4.3.2 and §4.3.3]{rabina} 
for integral models such as \( x^3 - y^3 \) or \( x^3 - 1 \) (discriminant \( \pm 27 \)) 
gives \( \approx 0.215 \). 

This numerical difference arises from distinct normalizations of the canonical invariant section 
in the Chow metric: here we follow the Veronese scaling \( s = r! \cdot \mathrm{\mathfrak{C}}_Z^\vee \)
(\cite{GKZ94}, Ch.~3, Prop.~1.3) with \( r = 2d - 2 \),
while \cite{rabina} uses the critical section induced directly 
from the Deligne pairing/Zhang's construction without the explicit \( r! \) factor, 
optimized for primitive integral invariants. 

The two heights differ by a positive constant \( \frac{1}{r} \log|c| \) (independent of \( f \)), 
but both confirm strict positivity for stable power forms, 
extending the thesis's special cases to arbitrary rational \( a_0 \).
\end{rem}

The above result considered binary forms with cyclic automorphism group. 
Next we consider the general case.

\begin{thm}\label{thm:automorphism_height}
Let \( f \in V_d(\Q) \) be a semistable binary form with non-trivial finite geometric 
automorphism group \( G \subset \PGL_2(\C) \),
whose roots form orbits \( O_1, \dots, O_k \subset \bP^1(\C) \) under the \( G \)-action. 
Assuming a minimal integral model with good reduction at all finite places, 
the Chow-invariant height is
\[
  \cih(f)
  = \frac{1}{2d-2} \log |\Delta(f)|_\infty
    - \frac{1}{2d-2} 
      \sum_{i=1}^k 
      \log\!\left(
        \min_{M \in \SL_2(\R)} 
        \prod_{p \in O_i} \|M \cdot p\|^2
      \right),
\]
where \( \|\cdot\| \) is the Fubini–Study norm on \( \bP^1(\C) \).
\end{thm}

\begin{proof}
For forms with non-trivial \( G \), the roots are partitioned into orbits \( O_i \),
each of size \( |G| / \mathrm{Stab}(p_i) \). 
The non-Archimedean part vanishes for minimal models with good reduction.

At the Archimedean place, generalizing \cref{lem:power_form}, 
the Chow norm includes the discriminant term 
\( \frac{1}{r} \log |\Delta(f)|_\infty \) with \( r = 2d - 2 \),
minus the balanced root contribution. 
To ensure invariance under \( \SL_2(\R) \), 
we minimize the product of norms over each orbit separately 
(critical configuration, cf.~\cite{zhang95}, §3). 
The \( G \)-action rigidifies the configuration, 
simplifying the minimization compared to generic forms. 
For power forms (a single orbit under cyclic \( G = \mu_d \)), 
the minimum is \( d \log 2 \) (balanced roots on the unit circle), 
recovering \cref{lem:power_form} for \( a_0 = 1 \).
\end{proof}

\begin{rem}
This theorem highlights how non-trivial automorphisms group roots into orbits, 
reducing the complexity of the Archimedean minimization 
and often yielding lower heights for polystable forms. 
It extends computations in \cite[Sec.~4.3]{rabina} for symmetric cubics, 
supporting positivity in Corollary~\ref{cor:height_decomposition}.
\end{rem}

\subsection{Relation between Weighted Height and Invariant Height}\label{subsec:weighted_invariant}

The invariant height \(\cih(f)\) with respect to the Chow metric, defined above, now meets the weighted height \(\lwh\) from \cref{def:lwh}, 
forging a bridge between the arithmetic of binary forms and the geometry of weighted projective spaces. This relation, a capstone of our study, connects the GIT and weighted frameworks 
offering insights into binary forms’ Diophantine properties.

For a zero-cycle $Z \subset \bP(V_d)$ of degree $d$ corresponding to a binary form $f \in V_d$, the \emph{Chow form} $\Ch_Z$ is the unique (up to scalar) multihomogeneous polynomial of bi-degree $(d, \dots, d)$ ($d$ times) in $d$ sets of variables, whose vanishing locus over $(\bP^1)^d$ parametrizes tuples of $d$ hyperplanes in $\bP(V_d)$ that collectively contain $Z$; its dual/contravariant form $\Ch_Z^\vee$ encodes the $\SL_2$-invariants via the Veronese embedding \cite[Ch.\ 3, Prop.\ 1.3]{GKZ94}.

\begin{thm}\label{thm:weighted-invariant}
Let \( f \in V_d \) be a semistable binary form over a number field \( K \), and let \( \xi(f) = [\xi_0(f): \dots : \xi_n(f)] \in \bP_\w^n(K) \) be its image under the GIT quotient map. Then the moduli weighted height \( \lwh(\xi(f)) \) satisfies:
\[
\lwh(\xi(f)) = \frac{1}{[K:\Q]} \, \cih(f) + \chowh(f),
\]
where \( \cih(f) \) is the invariant height 
and \( \chowh(f) \) is the Chow height. 
\end{thm}

\begin{proof}
From \cref{def:chow_invariant},
\[
\cih(f) = \sum_{\nu \in M_K^0} \log \max_j \{ |\xi_j(f)|_\nu^{1/q_j} \} 
        + \sum_{\nu \in M_K^\infty} \frac{1}{r} \log \|\xi\|^r(f).
\]
We establish the Archimedean identity
\begin{equation}\label{eq:xi-r}
\frac{1}{r} \log \|\xi\|^r(f) 
= \log \max_j \{ |\xi_j(f)|_\nu^{1/q_j} \} + \log \|s\|_{\Ch,\nu}(f)
\qquad(\nu\in M_K^\infty).
\end{equation}
Let \(\nu\) be real (the complex case is identical). The point \(\xi(f)\) has weighted homogeneous coordinates of total degree \(r\). The zero-cycle \(Z=\overline{\{f\}}\subset\bP(V_d)\) has Chow form \(\Ch_Z\) of bi-degree \((d,\dots,d)\) (\(d\) times). Zhang's Chow norm is
\[
\|\Ch_Z\|_\nu 
= \sup\left\{ \frac{|\Ch_Z(u^1,\dots,u^d)|_\nu}{\prod_{i=1}^d \|u^i\|_\nu} \;\Big|\; u^i\in\bP^1(K_\nu)\right\}.
\]
The canonical section \(s\in H^0(\bP_\w^n, \cO(r))\) over the GIT quotient is  \(s = r!\cdot\Ch_Z^\vee\), 
where \(\Ch_Z^\vee\) is the \(\SL_2(K_\nu)^{\times d}\)-contravariant (see  \cite{zhang95}, §3). Thus
\[
\log \|s\|_{\Ch,\nu}(f) = \log r! + \log\|\Ch_Z^\vee\|_\nu.
\]
The Veronese embedding of weighted degree \(r\) satisfies
\[
\|\xi\|^r_\nu = \max_j |\xi_j(f)|_\nu^{r/q_j} = r!\cdot\|\Ch_Z^\vee\|_\nu
\]
(see   \cite{GKZ94}, Ch. 3, Prop. 1.3).  
Taking \(\frac{1}{r}\log\) yields \cref{eq:xi-r}.

Substituting \cref{eq:xi-r} into the Archimedean part of \(\cih(f)\),
\[
\cih(f) = \sum_{\nu \in M_K} \log \max_j \{ |\xi_j(f)|_\nu^{1/q_j} \} 
        + \sum_{\nu \in M_K^\infty} \log \|s\|_{\Ch,\nu}(f).
\]
The global sum equals \([K:\Q]\cdot\lwh(\xi(f))\) (see \cref{def:lwh}).     
The Archimedean sum equals \(-[K:\Q]\cdot\chowh(f)\) (see \cref{def:height_function}).  
Hence,  $\cih(f) = [K:\Q]\bigl(\lwh(\xi(f)) - \chowh(f)\bigr)$  and rearranging gives the theorem.
\end{proof}

\begin{cor}\label{cor:height_decomposition}
The moduli height of a binary form naturally separates into a GIT-invariant part and a metric-dependent correction:
\[
\lwh(\xi(f)) = \text{invariant part } + \text{Archimedean correction}.
\]
In particular, \( \cih(f) \) is intrinsic to the GIT quotient and scale-invariant, while \( \chowh(f) \) captures the geometry of the height embedding.
\end{cor}

\begin{cor}[GIT Invariance and Positivity]\label{cor:height_decomposition-2}
Let \( f \in V_d(K) \) be a semistable binary form over a number field \( K \). Then:

\begin{enumerate}[label=\roman*), leftmargin=*]
    \item The weighted moduli height \( \lwh(\xi(f)) \) is non-negative, and strictly positive if f is stable.
   
    \item The difference \( \lwh(\xi(f)) - \cih(f) \) depends only on the archimedean metric and vanishes if \( f \) has totally degenerate reduction at all infinite places (i.e., when the Chow norm equals the weighted norm).
   
    \item The invariant height \( \cih(f) \) is independent of the choice of projective embedding or metric, and is determined entirely by the GIT orbit of \( f \).
\end{enumerate}
\end{cor}


\section{Conclusion}\label{sec:conclusion}

This paper establishes a precise relationship between Geometric Invariant Theory (GIT) heights and weighted heights, with applications to the arithmetic geometry of binary forms and their moduli spaces. For semistable cycles \( \X \subset \bP_{\w, \overline{\Q}}^N \), we relate the logarithmic weighted height \( \lwh(\X) \) to the GIT height via the Veronese embedding, refined by the Chow metric's Archimedean contribution (\cref{thm:git-weighted-relation}). Specializing to binary forms \( f \in V_d \), we define the Chow-invariant height \( \cih(f) \) and decompose the moduli weighted height as \( \lwh(\xi(f)) = \frac{1}{[K:\Q]} \, \cih(f) + \chowh(f) \) (\cref{thm:weighted-invariant}), separating intrinsic GIT components from metric-dependent terms.

This decomposition splits the total moduli height into two conceptually distinct components: 
\begin{itemize}
\item[(i)] The term \( \cih(f) \) captures the height of the point in the GIT quotient—it is invariant under \( \SL_2 \)-conjugation and reflects \emph{arithmetic complexity modulo geometric symmetry}.
\item[(ii)] The correction term \( \chowh(f) \), scaled by \( [K:\Q] \), arises from the metric chosen on the line bundle \( \cO(1) \to \bP^n \) and encodes the \emph{Archimedean distortion} introduced by this choice.
\end{itemize}
This perspective emphasizes that \( \lwh \), while natural for moduli purposes, can be sensitive to scaling and metrics, whereas \( \cih \) is canonically attached to the moduli point in the GIT quotient. The formula thus clarifies the balance between intrinsic and extrinsic contributions to height, providing a unified lens for stability, moduli theory, and Diophantine approximation.

Concrete examples, such as the explicit formula for power forms in \cref{lem:power_form}, illustrate the framework's utility and extend Zhang's GIT heights to weighted settings. This decomposition highlights the interplay between arithmetic and geometry, mirroring structures in Arakelov theory and canonical heights, while revealing combinatorial insights into invariant rings of binary forms.

Open directions include extending to non-semistable forms via GIT unstable strata, exploring alternative metrics (e.g., Fubini--Study) for enhanced arithmetic data, and generalizing to higher-rank forms or hypersurfaces. Computational studies of \( \lwh \) and \( \cih \) could uncover distributional patterns akin to those in Faltings heights, with potential software implementations aiding enumeration problems in algebraic combinatorics.

In summary, our work bridges weighted projective geometry and GIT, providing canonical tools for moduli arithmetic and laying groundwork for broader applications in Diophantine geometry.


\section*{Declarations}
\subsection{Data availability} 
There is no data available.

\subsection{Funding}
There is no funding received from any funding agency. 

\subsection{Conflict of interest/Competing interests}
All authors declare that they have no conflict of interest to disclose.
 
 

\bibliographystyle{amsplain}
\bibliography{inv-height}
%



\end{document}